\documentclass[10pt]{amsart}

\usepackage[T1]{fontenc}  
\usepackage[utf8]{inputenc} 
\usepackage{longtable}
\usepackage{amssymb}
\usepackage{amsthm}
\usepackage{amsmath}
\usepackage{wasysym}
\usepackage{dcpic, pictex}
\usepackage{bbm}
\usepackage{enumerate}
\usepackage{amsfonts}
\usepackage{amsthm}
\usepackage{amsmath}
\usepackage{times}
\usepackage{amscd}
\usepackage[mathscr]{eucal}
\usepackage{indentfirst}
\usepackage{pict2e}
\usepackage{epic}
\usepackage{epstopdf} 
\usepackage{verbatim}
\usepackage{cancel}
\usepackage[foot]{amsaddr}
\usepackage{lipsum}
\usepackage{mathrsfs}
\usepackage{bm}

\usepackage{subcaption}
\usepackage{graphicx}

\usepackage{
amsfonts, 
latexsym, 
enumerate,
}
\usepackage[english]{babel}
\allowdisplaybreaks
\makeindex
\newtheorem{theorem}{Theorem}[section]
\newtheorem{lemma}[theorem]{Lemma}
\newtheorem{proposition}[theorem]{Proposition}

\newtheorem{corollary}[theorem]{Corollary}
\newtheorem{remark}[theorem]{Remark}
\newtheorem{assumption}[theorem]{Assumption}
\usepackage[a4paper,top=3.5cm,bottom=3.5cm,left=3cm,right=3cm]{geometry}
\numberwithin{equation}{section}

\allowdisplaybreaks

\usepackage{ bbold }
\usepackage{xcolor}

\newcommand{\nc}{\normalcolor}
\newcommand{\dif}{\mathrm{d}}
\newcommand{\E}{\mathbf{E}}

\newcommand{\R}{\mathbf{R}}
\newcommand{\C}{\mathbf{C}}
\newcommand{\DD}{\mathbf{D}}

\newcommand{\N}{\mathbf{N}}

\newcommand{\Cov}{\mathrm{Cov}}
\newcommand{\Var}{\mathrm{Var}}

\newcommand{\ii}{\mathrm{i}}
\newcommand{\ee}{\mathrm{e}}

\newcommand{\mt}{\mathfrak{t}}

\usepackage{hyperref}

\title[Minor process]{The non-Hermitian minor process}

\date{\today}

\begin{document}

\maketitle

\vspace{0.25cm}

\renewcommand{\thefootnote}{\fnsymbol{footnote}}

\noindent
\mbox{}%
\hfill%
\begin{minipage}{0.19\textwidth}
	\centering
	{Giorgio Cipolloni}\footnotemark[1]\\
	\footnotesize{\textit{cipolloni@axp.mat.uniroma2.it}}
\end{minipage}
\hfill%
\begin{minipage}{0.19\textwidth}
	\centering
	{L\'aszl\'o Erd\H{o}s}\footnotemark[2]\\
	\footnotesize{\textit{lerdos@ist.ac.at}}
\end{minipage}
\hfill%
\begin{minipage}{0.19\textwidth}
	\centering
	{Oleksii Kolupaiev}\footnotemark[2]\\
	\footnotesize{\textit{okolupaiev@ist.ac.at}}
\end{minipage}
\hfill%
\mbox{}%
\footnotetext[1]{University of Rome Tor Vergata, Via della Ricerca Scientifica 1, 00133 Rome, Italy.}
\footnotetext[2]{Institute of Science and Technology Austria, Am Campus 1, 3400 Klosterneuburg, Austria. 
}

\renewcommand*{\thefootnote}{\arabic{footnote}}
\vspace{0.25cm}

\begin{abstract} 
We show that the log-determinant of leading principal minors of large non-Hermitian random matrices converges in distribution to a 2+1 dimensional Gaussian field, which is logarithmically correlated for the parabolic distance, 
reminiscent to the Edwards-Wilkinson universality class.
\end{abstract}
\vspace{0.15cm}


\vspace{0.25cm}
\normalsize


\section{Introduction}

We consider an $N\times N$ matrix $X$ with independent, identically distributed (i.i.d.) centered entries and its 
\emph{leading principal  minors}, i.e. square sub-matrices in the lower right corner\footnote{We deviate from the usual terminology where leading principal minors refer to minors in the top left corner.}. We study the joint distribution of their eigenvalues through the correlation structure of their linear eigenvalue statistics (see \eqref{eq:linstatmin} below), in particular, we describe how the different minor levels influence the answer. To motivate this line of study, we first summarize the analogous results for the Hermitian case, where there has been a lot of progress. Then, we will explain our result on minors of non-Hermitian random matrices. In particular, the main novelty of this work is to investigate this natural question in the non-Hermitian setting. Somewhat surprisingly, our result is new even for the \emph{Ginibre ensemble}, i.e. for the Gaussian case, where otherwise many explicit formulas are available.

Let now $W$ be a \emph{Wigner matrix}, i.e. an $N\times N$ Hermitian random matrix with i.i.d. centered entries (up to the Hermitian symmetry) with $\E|W_{ij}|^2=N^{-1}$, and denote its eigenvalues by $\lambda_1,\dots, \lambda_N$. It is well known since \cite{wigner1958distribution} that the (random) empirical measure $\frac{1}{N}\sum_i\delta_{\lambda_i}$ converges to a limiting deterministic measure $\mu_{\mathrm{sc}}(\dif x)$ (i.e. the \emph{semicircular distribution}). It is then natural to ask for the size of the fluctuations around this deterministic limit. In \cite{Anderson06, Bai04, Johansson04, Lytova09, Sinai98} (see also \cite{Shcherbina11, Sosoe13} for later results) it was discovered\footnote{In fact, this unusually small fluctuations were discovered earlier by Dyson and Mehta \cite{dyson1963statistical} when studying the fluctuations of the eigenvalues of Haar unitary random matrices (CUE ensemble).} that the fluctuations around $\mu_{\mathrm{sc}}(\dif x)$ are of size $N^{-1}$,  i.e. unusually small. This is due to the very strong correlation among the eigenvalues of $W$, if the eigenvalues were independent or weakly correlated, the usual CLT would predict a much larger fluctuation of order $N^{-1/2}$. The fluctuation, however,  is still Gaussian on this small scale. More precisely, we have
\[
L_N(f):=  \sum_i f(\lambda_i)-\E \sum_i f(\lambda_i)\stackrel{N\to\infty}{\Longrightarrow} \mathcal{N}(0,V_f),
\]
for any regular test function $f$, with an explicit $V_f>0$. In particular, the limiting field can be identified as a one-dimensional logarithmically correlated (or \emph{log-correlated}, in short) field. Instead of a single matrix $W$ it is then natural to consider the family of its leading principal minors\footnote{Here $W^{(k)}$ denotes the $k\times k$ matrix obtained from $W$ by removing its first $N-k$ rows and columns.} $W^{(k)}$ for different $k$'s to study the joint distribution of their eigenvalues, and show that these fluctuations converge to a two-dimensional Gaussian field. In \cite{Bor} Borodin considered the joint spectral statistics of an $N\times N$ Wigner matrix $W=W^{(N)}$ and its minors $W^{(xN)}$, for $x\in (0,1]$, where we assume that $xN$ is an integer for notational simplicity. In this work, Borodin showed that $\sum_i f(\lambda_i^x,x)-\E \sum_i f(\lambda_i^x,x)$, with $\lambda_i^x\in\mathrm{Spec}(W^{(xN)})$ and $x\in (0,1]$, converges to  the two-dimensional \emph{Gaussian Free Field (GFF)}, which can be realized by $h_N(E, x) =\log|\mathrm{det}(W^{(xN)}-E-\ii 0)|$ in the sense that $L_N(f)=\langle h_N, f\rangle_{H^1}$ after a convenient change of variables. Since this fundamental work of Borodin there has been a great interest in joint fluctuations of analogous two-dimensional processes\footnote{We refer to the introduction of \cite{bourgade2024fluctuations} for a detailed description of all the results connecting random matrices and logarithmically correlated fields. Here we focus only on the results concering minors or analogous processes defined for certain $\beta$-ensembles.}, e.g. for sample covariance matrices \cite{Dumitriu18}, $\beta$-Jacobi ensembles \cite{beta_corners15}, the adjacency matrix of random regular graphs \cite{Johnson14, Ganguly20}, and random tilings \cite{borot2026macroscopic} and references therein. We also mention a related line of work studying the difference of linear eigenvalue statistics of a matrix and its  immediate minor obtained by removing one row and column, which was initiated in \cite{beta_corners18} for the $\beta$-Jacobi ensemble and then, motivated by this work, studied for Wigner matrices \cite{Dominik18} and for sample covariance matrices \cite{SampleCov20}. The outcome can be interpreted as convergence to the \emph{derivative} of a GFF. Furthermore, \cite{BorII} studied the joint correlations of eigenvalues of minors for a dynamical version of Wigner matrices giving rise to a three dimensional log-correlated field (with respect to the Euclidean distance) whose certain two dimensional slices are GFF.

While there has been great interest and success in studying the joint eigenvalue fluctuations of minors of Hermitian random matrices, surprisingly, no results are known in the non-Hermitian case, not even for the \emph{Ginibre ensemble}\footnote{We say that an $N\times N$ matrix $X$ belongs to the real/complex Ginibre ensemble if the entries $\sqrt{N}X_{ij}$ are distributed as i.i.d. standard real/complex Gaussians.}, i.e. when the matrix entries are independent Gaussian. While all our results in this paper hold for general non-Hermitian matrices with either real or complex i.i.d. entry distribution, for simplicity of this introduction, here we present only the complex Gaussian model. For the results in full generality, see Section~\ref{sec:results} below. We thus define the \emph{Ginibre minor process}, i.e. we consider an $N\times N$ complex Ginibre matrix $X$ and its leading principal minors\footnote{We again assume that $xN$ is an integer, for notational simplicity.} $X^{(xN)}$, for $x\in (0,1]$. In particular, for any smooth function $f:\C\to\C$,  we study the joint fluctuations of the (centered) linear statistics 
\begin{equation}
\label{eq:linstatmin}
L^{(xN)}(f):=\sum_{i=1}^{xN}f\big(\sigma_i^{(xN)}\big)-\E\sum_{i=1}^{xN}f\big(\sigma_i^{(xN)}\big),
\end{equation}
where $\sigma_i^{(xN)}\in \mathrm{Spec}(X^{(xN)})$, for different $x\in (0,1]$. 

\begin{figure}[h]
\begin{center}
\begin{subfigure}{0.45\textwidth}
\includegraphics[height=7cm]{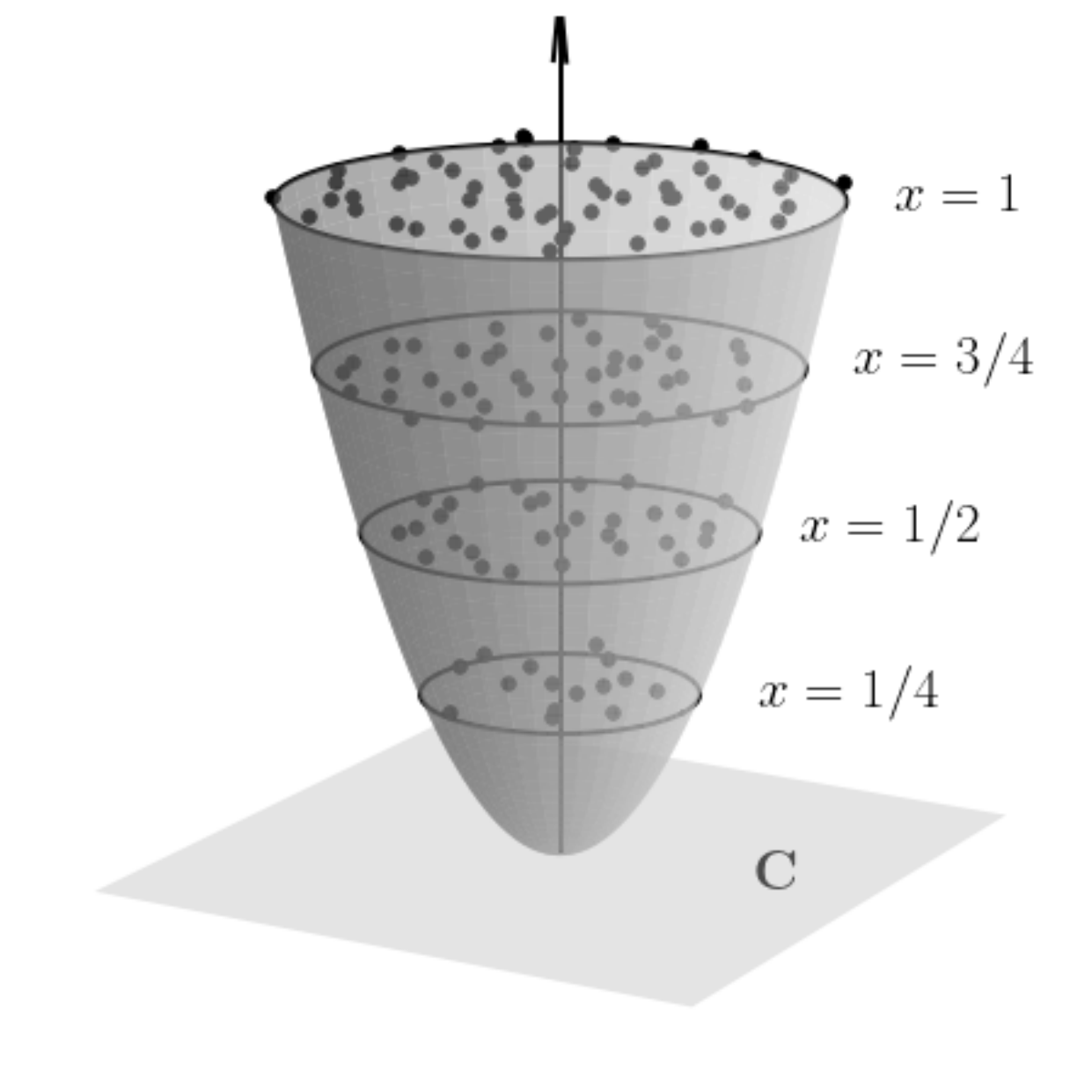}
\caption{}
\end{subfigure}%
\hfill
\begin{subfigure}{0.45\textwidth}
\includegraphics[height=7cm]{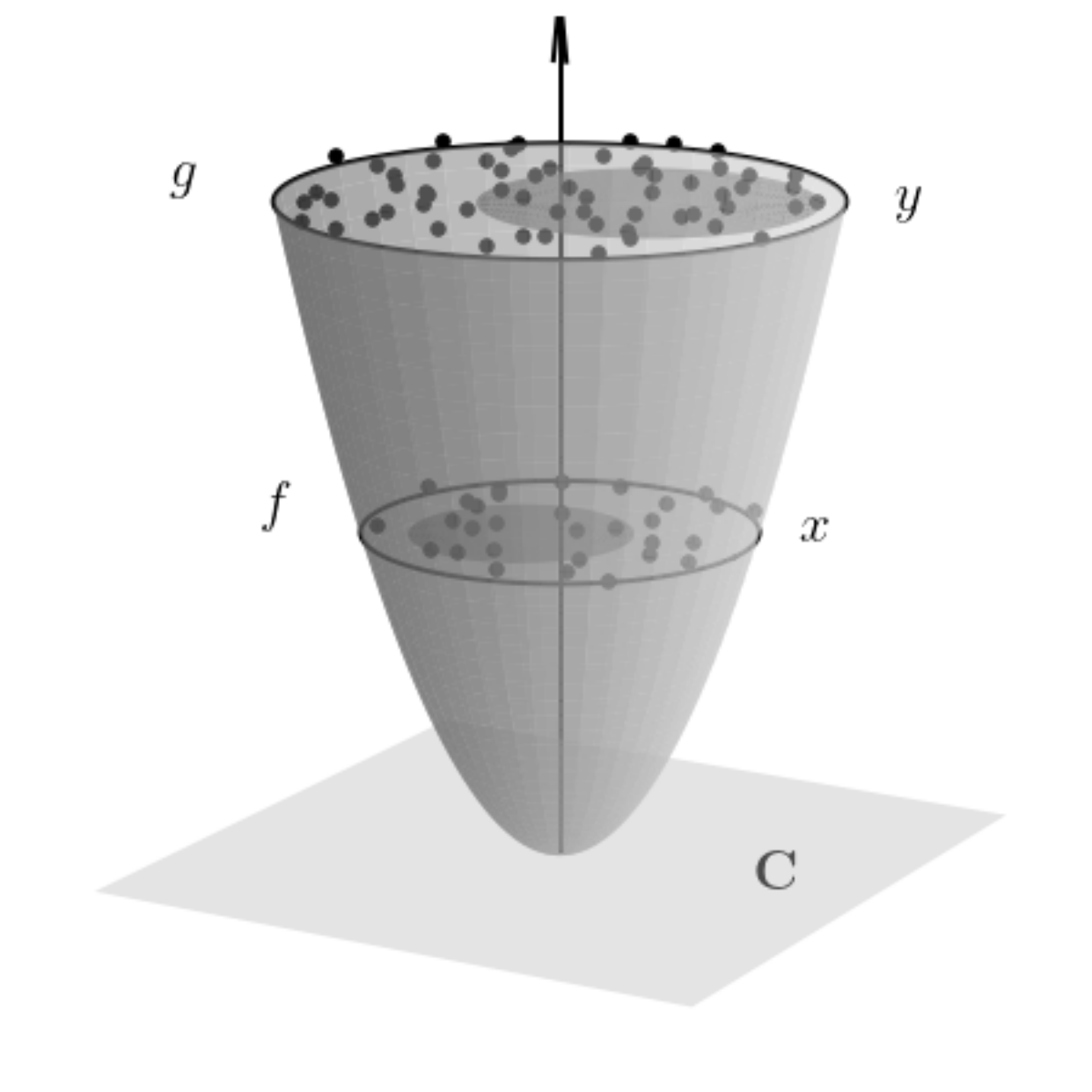}
\caption{}
\end{subfigure}
\end{center}
\caption{In the left panel we depict the eigenvalues of $X^{(xN)}$ at several levels of the non-Hermitian minor process. By Girko's circular law, almost all of these eigenvalues are contained inside of the parabolic surface $|z|^2=x$. The right panel illustrates the set-up of Theorem~\ref{theo:intro}. The dark grey disks chosen at the levels $x$ and $y$ with $x\le y$, represent the supports of the test functions $f$ and $g$, respectively. }
\label{fig:minor_process}
\end{figure}

\begin{theorem}\label{theo:intro}
Consider the complex Ginibre minor process $X^{(xN)}$, and fix any $x,y\in (0,1]$ with $x\le y$. Let $f,g$ be smooth functions supported in the open disks\footnote{Here $\DD_x:=\{z\in\C:|z|\le \sqrt{x}\}$.} $\DD_x$, $\DD_y$, respectively. Then, the linear statistics $L^{(xN)}(f), L^{(yN)}(g)$ converge jointly to centered Gaussian random variables $L(x,f), L(y,g)$, with covariance
\begin{equation}
\E L(x,f) L(y,g)= -\frac{1}{\pi^2} \int_{\DD_x} \!\!\dif^2 z\int_{\DD_y}\!\!\dif^2 w\,\, \overline{\partial_z f(z)}\partial_w g(w) \partial_z\partial_{\overline{w}} \log K\big((x,z),(y,w)\big),
\label{eq:intro_main}
\end{equation}
where $K$ is given by\footnote{Later in \eqref{eq:def_K} we extend the definition of $K$ to general $x,y\in(0,1]$.}
\begin{equation}
K\big((x,z),(y,w)\big):=|z-w|^2+|x-y|\left(1-\frac{|w|^2}{y}\right), \qquad z\in \DD_x, \; w\in \DD_y.
\end{equation}
\end{theorem}

See Figure~\ref{fig:minor_process} for the illustration of the set-up of Theorem~\ref{theo:intro}.

In the case $x=y$ the derivatives of $\log K$ in the rhs. of \eqref{eq:intro_main} are understood in the sense of distributions. Note that we have $K\big((x,z),(y,w)\big)\sim d(z,w,|x-y|)$, where $d(z,w,|x-y|):=|z-w|^2+|x-y|$ denotes the parabolic distance between two points $(x,z)$ and $(y,w)$ inside a natural three dimensional cone. This shows that the fluctuation structure of the  \emph{Ginibre minor process} is analogous to the Edwards-Wilkinson universality class describing correlated evolution processes if the minor variable $x$ is interpreted as time. Actually, for the Ornstein-Uhlenbeck dynamics on Ginibre matrices the Edwards-Wilkinson fluctuation was very recently proven in \cite{bourgade2024fluctuations}. While the statement of our current result is reminiscent to this genuinely time evolving model, the interrelation between spectra of minors and time evolving matrices is of completely different origin. In particular, our result shows that the eigenvalue of a Ginibre matrix and its minors are strongly correlated (in fact they form a $2+1$-dimensional log-correlated field). While this strong correlation is natural for Hermitian matrices, where the interlacing property of the eigenvalues holds, this is not at all obvious in the non-Hermitian setting.

For the proof we rely on Girko's Hermitization formula (see \cite{Girko84}, \cite{TaoVu15})
\begin{equation}
\begin{split}
\sum_{\sigma\in\mathrm{Spec}(X)}f(\sigma)&=-\frac{1}{4\pi}\int\Delta f(z)\int_0^\infty \Im \mathrm{Tr} G^z(\ii\eta)\dif \eta\dif^2z, \\
H^z:&=\left(\begin{matrix}
0 & X-z \\
(X-z)^* & 0
\end{matrix}\right), \qquad\quad G^z(\ii\eta):=(H^z-\ii\eta)^{-1},
\end{split}
\label{eq:Girko_intro}
\end{equation}
following the approach paved in \cite{macroCLT_complex}; we refer the interested reader to \cite[Section~3]{macroCLT_complex} for a detailed description of this proof strategy. Compared to \cite{macroCLT_complex}, the current set-up features a new challenge, since we also consider correlations between different levels of the minor process. To illustrate the essence of this novelty, let us compute $\Cov\left(L^{(k_1)}(f_1), L^{(k_2)}(f_2)\right)$ using \eqref{eq:Girko_intro}. We encounter the covariance
\begin{equation}
\Cov\left(\mathrm{Tr} G^{(k_1),z_1}(\ii\eta_1),\mathrm{Tr} G^{(k_2),z_2}(\ii\eta_2)\right),
\label{eq:intro_Cov}
\end{equation}
where the resolvent\footnote{Later in \eqref{eq:def_hermitization} we introduce a more detailed notation for this resolvent.} $G^{(k),z}(\ii\eta)$ is defined as in the second line of \eqref{eq:Girko_intro}, but with $X$ replaced by its minor $X^{(k)}$. This covariance in the regime $\eta_1,\eta_2\gg N^{-1}$ is the central object of our study since it determines the limiting correlation in the rhs. of \eqref{eq:intro_main}. For $k_1=k_2$, \eqref{eq:intro_Cov} can be computed (to leading order) in terms of \emph{averaged two-resolvent chains}, $\mathrm{Tr} [G_1B_1G_2B_2]$, where $G_j:=G^{(k),z_j}(\ii\eta_j)\in\C^{(2k)\times (2k)}$ and $B_1,B_2\in\C^{(2k)\times (2k)}$ are deterministic matrices. This type of quantities are well-controlled by the so called \emph{two-resolvent local laws} (see e.g. \cite{Multi_res_llaws}), which show that $G_1B_1G_2$ concentrates around its deterministic counterpart with very high probability. For general $k_1\neq k_2$, a similar calculation still yields averaged two-resolvent  chains, however $B_1,B_2$ are no longer square matrices, since $G_1$ and $G_2$ are of different dimensions. More precisely, for $k_1<k_2$ we have $B_1=\widetilde{B}_1P^*$ and $B_1=P\widetilde{B}_2$. Here $\widetilde{B}_1$, $\widetilde{B}_2$ are square matrices and $P$ is the structural projection from the space of larger matrices to the space of smaller once. By \emph{structural} we mean that $P$ maps the Hermitization of a larger minor of $X$ to the Hermitization of a smaller one, thereby preserving the structure of the minor process.

One way to handle a resolvent chain of the form $G_1\widetilde{B}_1P^*G_2P\widetilde{B}_2$ is to perform a block-inversion in $G_2$ via the Schur complement formula and rewrite $P^*G_2P$ as a resolvent of a smaller matrix with an additional deformation. For instance, this approach was used in \cite[Section~4]{minor} in the context of Wigner minor process. The main drawback of the Schur complement formula approach  is that the resulting deformation is genuinely non-Hermitian, even if the original setting is Hermitian. In \cite{minor} we only needed an upper bound on a two-resolvent chain which was accessible by a fairly simple direct estimate on the non-Hermitian deformation. However, in the current work we need to compute the leading term.  As a consequence, using the Schur formula would require a second Hermitization step,  effectively doubling the dimension once more and leading to a quite complicated system of $4\times 4$ block matrices. Instead of choosing this path, we prove the local law for $G_1\widetilde{B}_1P^*G_2P\widetilde{B}_2$ by working directly with this type of quantities.  This approach provides a flexible way to handle resolvents of different sizes, and may therefore be useful in other models where such resolvents arise.

\subsection*{Acknowledgement} The research of G. C. is supported by the Italian Ministry of University and Research (MUR) - Fondo Italiano per la Scienza (FIS3) - 2024 Call, project UBLOCO, CUP F53C25000940001, and also partially supported by the MUR Excellence Department Project MatMod@TOV awarded to the Department of Mathematics, University of Rome Tor Vergata, CUP E83C23000330006. Additionally, G.C. thanks INdAM (Istituto Nazionale di Alta Matematica “Francesco Severi”) and the group GNFM. The research of L. E. and O. K. is supported by the ERC Advanced Grant ``RMTBeyond'' No.~101020331.

\subsection*{Notations and conventions}
We set $[k] := \{1, ... , k\}$ for a positive integer $k \in \N$. The sets of real and complex numbers are denoted by $\R$ and $\C$, respectively, while $\mathbf{D}\subset \C$ stands for the open unit disk. For positive quantities $f, g$ we write $f \lesssim g$, $f \gtrsim g$, to denote that $f \le C g$ and $f \ge c g$, respectively, for some $N$-independent constants $c, C > 0$ that depend only on the basic control parameters of the model in Assumption~\ref{ass:chi} below. In informal explanations we sometimes use the notation $f\ll g$, which indicates that $f$ is "much smaller" than $g$.

We use the notation $\langle A \rangle := d^{-1} \mathrm{Tr}(A)$, $d \in \N$, for the normalized trace of a $d \times d$ matrix $A$. Additionally, we denote its transpose by $A^\mt$. Matrix entries are indexed by lower case Roman letters $a, b, c , ... ,i,j,k,... $ from the beginning or the middle of the alphabet. We denote vectors by bold-faced lower case Roman letters $\boldsymbol{x}, \boldsymbol{y} \in \C^{d}$, for some $d \in \N$. Moreover, for vectors $\boldsymbol{x}, \boldsymbol{y} \in \C^{d}$ and a matrix $A\in\C^{d\times d}$ we define
 \begin{equation*}
	\langle \boldsymbol{x}, \boldsymbol{y} \rangle := \sum_{i=1}^d \bar{x}_i y_i\,, 
	\qquad\quad A_{\boldsymbol{x} \boldsymbol{y}} := \langle \boldsymbol{x}, A \boldsymbol{y} \rangle\,. 
\end{equation*}

For a complex variable $z=x+\ii y$ we set
\begin{equation*}
\partial_z:=\frac{\partial_x-\ii\partial_y}{2},\quad \partial_{\overline{z}}:=\frac{\partial_x+\ii\partial_y}{2},\quad \dif z:=\dif x+\ii\dif y,\quad \dif\overline{z}:=\dif x-\ii\dif y,\quad \dif^2 z:=\frac{\ii}{2}\dif z\wedge\dif \overline{z}.
\end{equation*}  

The covariance of two complex-valued random variables $\zeta_1$ and $\zeta_2$ is denoted by
\begin{equation*}
\Cov(\zeta_1,\zeta_2):=\E\left[ (\overline{\zeta_1}-\E\overline{\zeta_1})(\zeta_2-\E\zeta_2)\right].
\end{equation*} 
We further denote the variance of a complex-valued random variable $\zeta$ by $\Var[\zeta]:=\Cov(\zeta,\zeta)$. For a pair of real or complex-valued stochastic processes $X=X(t)$ and $Y=Y(t)$, we denote their covariation process by~$[X,Y]_t$.

Finally, we will use the concept  of \emph{with very high probability},  meaning that for any fixed $D > 0$, the probability of an $N$-dependent event is bigger than $1 - N^{-D}$ for all $N \ge N_0(D)$. We will use the convention that $\xi > 0$ denotes an arbitrarily small positive exponent, independent of $N$. Moreover, we introduce the common notion of \emph{stochastic domination} (see, e.g., \cite{loc_sc_gen}): For two families
\begin{equation*}
	X = \left(X^{(N)}(u) \mid N \in \N, u \in U^{(N)}\right) \quad \text{and} \quad Y = \left(Y^{(N)}(u) \mid N \in \N, u \in U^{(N)}\right)
\end{equation*}
of non-negative random variables indexed by $N$, and possibly a parameter $u$, we say that $X$ is stochastically dominated by $Y$, if for all $\epsilon, D >0$ we have 
\begin{equation*}
	\sup_{u \in U^{(N)}} \mathbf{P} \left[X^{(N)}(u) > N^\epsilon Y^{(N)}(u)\right] \le N^{-D}
\end{equation*}
for large enough $N \ge N_0(\epsilon, D)$. In this case we write $X \prec Y$. If for some complex family of random variables we have $\vert X \vert \prec Y$, we also write $X = \mathcal{O}_\prec(Y)$.

\section{Main results}\label{sec:results}

We consider real ($\beta=1$) or complex ($\beta=2$) \emph{i.i.d. matrices} $X$, i.e. $N\times N$ matrices whose entries $x_{ab}\stackrel{d}{=}N^{-1/2}\chi$ are independent and identically distributed. We impose the following assumptions on the (possibly $N$-dependent) random variable $\chi$.

\begin{assumption}\label{ass:chi}
\emph{(i)} The random variable $\chi$ is real-valued for $\beta=1$ and complex-valued for $\beta=2$. We assume that $\chi$ satisfies $\E \chi=0$, $\E |\chi|^2=1$ and additionally in the complex case $\E\chi^2=0$. For $\beta=1,2$ we also assume the existence of high moments, i.e. that there exist constants $C_p>0$, for any $p\in \N$, such that  
\begin{equation}
\E |\chi|^p\le C_p.
\label{eq:momass}
\end{equation}

\noindent\emph{(ii)} There exist a (small) $\mathfrak{a}>0$ and a (large) $\mathfrak{b}>0$ such that the probability density $\rho_\chi$ of $\chi$ satisfies\footnote{This mild regularity assumption is needed only to control the regime $\eta\le N^{-100}$ in the Girko's formula \eqref{eq:Girko_intro}. This assumption can be easily removed by standard methods introduced in \cite{TaoVu15} (see also \cite[Remark~2.2]{rightmost_ev}) and is kept here to make the presentation simpler.\nc}
\begin{equation}
\rho_\chi \in L^{1+\mathfrak{a}}(\R)\,\,\,\text{for}\,\,\, \beta=1,\quad \rho_\chi \in L^{1+\mathfrak{a}}(\C)\,\,\,\text{for}\,\,\, \beta=2\quad\text{and}\quad \|\rho_\chi\|_{1+\mathfrak{a}}\le N^\mathfrak{b}\,\,\, \text{for}\,\,\, \beta=1,2.
\label{eq:reg_ass}
\end{equation}
\end{assumption}

We call the $N$-independent constants $C_p, \mathfrak{a}$ and $\mathfrak{b}$ appearing in Assumption~\ref{ass:chi} the \emph{model parameters}. In the formulations of our main results, as well as in the intermediate technical results and their proofs, implicit constants in $\mathcal{O}(\cdot)$ and implicit exponents may depend on the model parameters, unless stated otherwise.

For an integer $k\in [N]$, we denote by $X^{(k)}$ the principal minor of $X$ formed by the last~$k$ rows and the last $k$ columns. Denote the eigenvalues of $X^{(k)}$ by $\{\sigma_i^{(k)}\}_{i\in [k]}\subset \C$. Collecting the eigenvalues of $X^{(k)}$ simultaneously for all $k\in[N]$, we obtain the point process $\Xi_N:=\{(k/N,\sigma_i^{(k)})\}_{k,i}$ in $[0,1]\times \C$. In the sequel we consistently denote $x:=k/N$, and use parameters $k\in [N]$ and $x\in [0,1]$ interchangeably. In the case of several parameters $k_i\in[N]$ we set $x_i:=k_i/N$.

For a \emph{test function} $F:[0,1]\times\C\to\C$, the \emph{centered linear statistics} of $\Xi_N$ associated to $F$ is defined by  
\begin{equation}
L_N(F):= \sum_{k\in [N]} \sum_{i\in [k]} F\big(k/N, \sigma_i^{(k)}\big)-\E \sum_{k\in [N]} \sum_{i\in [k]} F\big(k/N, \sigma_i^{(k)}\big).
\end{equation}
The correlation structure of $\Xi_N$ can be understood through the correlations between different linear statistics of this process. In this paper we mainly focus on the special case when $F$ is supported on one \emph{level} of the minor process $(k/N)\times \C$ for some $k\in[N]$. In this case $L_N(F)$ becomes the customary \emph{centered linear eigenvalue statistics} of $X^{(k)}$: 
\begin{equation}
L^{(k)}(f):=L_N(F)\quad\text{with}\quad f:=F(k/N,\cdot).
\end{equation}
Note that this choice of test functions is not restrictive: a general test function $F$ with a genuine 3-dimensional support can be represented as a sum of at most $N$ test functions, each supported on one level. Meanwhile, the precision of our results allows us to accommodate such sums, for more details see Corollary~\ref{cor:general_testfn}.

In this section we state our two main results concerning the \emph{macroscopic} and \emph{mesoscopic} scalings of test functions, respectively. In both regimes we show that centered linear statistics asymptotically have joint Gaussian distribution. In particular, their joint moments are determined by covariances of the form   
\begin{equation}
\Cov\left(L^{(k_1)}(f_1),L^{(k_2)}(f_2)\right),
\label{eq:Cov_intro}
\end{equation}
up to the leading order. 

\subsection{Macroscopic CLT}\label{sec:macroCLT}

First, we discuss the macroscopic regime, when the test functions $f_1,f_2$ in \eqref{eq:Cov_intro} do not depend on $N$. Our main result in this regime is as follows. 

\begin{theorem}[Macroscopic CLT]\label{theo:main} Let $X$ be a complex or real $N\times N$ i.i.d. matrix satisfying Assumption~\ref{ass:chi}. Fix (small) $c_0,\delta>0$ and an open bounded domain $\DD\subset \Omega\subset \C$. Then the following holds.

\emph(i) The process of centered linear statistics $L^{(k)}(f)$ converges to a centered Gaussian process $L(x,f)$ in the sense of moments: 
\begin{equation}
\E \prod_{i=1}^n \left(L^{(k_i)}(f_i)\right)^{p_i} \left(\overline{L^{(k_i)}(f_i)}\right)^{q_i} = \E \prod_{i=1}^n \left(L(x_i,f_i)\right)^{p_i}\left(\overline{L(x_i,f_i)}\right)^{q_i} + \mathcal{O}(N^{-c(\bm{p},\bm{q})})\,\,\,\text{with}\,\,\, x_i:=k_i/N,
\label{eq:main_macro}
\end{equation}
for any $N$-independent $n, p_i,q_i\in\N\cup\{0\}$, and for (possibly $N$-dependent) $x_i\in [c_0,1]$, for $i\in [n]$. The exponent $c(\bm{p},\bm{q})>0$ in \eqref{eq:main_macro} depends only on $\sum_{i=1}^n (p_i+q_i)$, and the implicit constant in $\mathcal{O}(\cdot)$ may depend only on $H^{2+\delta}$-norm of $f_i$, $i\in[n]$, and on the constants from Assumption~\ref{ass:chi}.

\emph{(ii)} The covariance structure of $\{L(x,f)\}$ is given by
\begin{equation}
\Cov\left(L(x_1,f_1),(x_2,f_2)\right)=C_\beta\left((x_1,f_1),(x_2,f_2)\right),
\label{eq:Cov_L}
\end{equation}
where the explicit sesquilinear form $C_\beta$ is defined below in \eqref{eq:C_beta_2}--\eqref{eq:C_beta_1}.
\end{theorem}

In the special case when all $k_i$'s are the same in \eqref{eq:main_macro}, we recover the results from \cite[Theorem~2.2]{macroCLT_real} and \cite[Theorem~2.2]{macroCLT_complex}. Additionally, Theorem~\ref{theo:main} implies the following result for test-functions defined on $[0,1]\times\C$ instead of just on one level of the minor process. The proof is presented in Appendix~\ref{sec:C_properties}.

\begin{corollary}[Test functions with 3-dimensional support]\label{cor:general_testfn} Let $X$ and $\Omega$ be as in Theorem~\ref{theo:main}. For test functions $F_1,F_2:[c_0,1]\times\C\to \C$ assume that $F_j(x,\cdot)$, $j=1,2$, satisfy the assumptions of Theorem~\ref{theo:main} uniformly in $x\in[c_0,1]$. Assume additionally that $F_1,F_2$ are $\alpha$-H{\"o}lder continuous in the first variable for some fixed $\alpha>0$. Then we have
\begin{equation}
N^{-2}\Cov\left(L_N(F_1),L_N(F_2)\right)=\int_{c_0}^1\dif x_1\int_{c_0}^1\dif x_2 C_\beta\Big(\big(x_1,F_1(x_1,\cdot)\big),\big(x_2,F_2(x_2,\cdot)\big)\Big)+\mathcal{O}(N^{-c}),
\label{eq:cor_general_testfn}
\end{equation} 
for some $N$-independent $c>0$. The implicit constant in $\mathcal{O}(\cdot)$ may depend only on $c_0$, $\sup_x \|F(x,\cdot)\|_{H^{2+\delta}(\Omega)}$, on the constant in the H{\"o}lder regularity condition and on model parameters from Assumption~\ref{ass:chi}. We also have that $N^{-1}L_N(F_1)$ and $N^{-1}L_N(F_2)$ are asymptotically jointly Gaussian distributed in the sense of \eqref{eq:main_macro}.
\end{corollary}

To justify the $N^{-2}$ scaling in the lhs. of \eqref{eq:cor_general_testfn} one needs to show that the integral in the rhs. of \eqref{eq:cor_general_testfn} does not vanish for $F_1=F_2$. This is done in the following remark.

\begin{remark}[Positivity of $C_\beta$] Let $F$ be any test function satisfying the assumptions of Corollary~\ref{cor:general_testfn}. Owing to $\Var[L_N(F)]\ge 0$, Corollary~\ref{cor:general_testfn} applied for $F_1=F_2=F$ yields that the sesquilinear form $C_\beta$ is non-negative:
\begin{equation}
\int_{c_0}^1\dif x_1\int_{c_0}^1\dif x_2 C_\beta\Big(\big(x_1,F(x_1,\cdot)\big),\big(x_2,F(x_2,\cdot)\big)\Big)\ge 0.
\label{eq:rem_C_positivity}
\end{equation}
The following argument shows that $C_\beta$ is strictly positive at least when the support of $F$ is mesoscopic or sufficiently small on the macroscopic scale. For $\ell>0$, $x_0\in (c_0,1)$ and $|z_0|<\sqrt{x_0}$ define 
\begin{equation}
F_\ell(x,z):=\ell^{-2} F\left(\ell^{-2}(x-x_0), \ell^{-1}(z-z_0)\right).
\end{equation}
We view $C_\beta$ as a perturbation around its mesoscopic version (see \eqref{eq:C_beta_meso} below), which is known to be strictly positive from \cite[Section~5.5]{bourgade2024fluctuations}. Thus, an elementary perturbative calculation gives that the lhs. of \eqref{eq:rem_C_positivity} evaluated at $F_\ell$ is a positive quantity of order 1 for sufficiently small $\ell>0$ (which may be chosen independently of $N$).
\end{remark}

Now we present an explicit formula for $C_\beta$ emerging in \eqref{eq:Cov_L}. First, for any $x_1,x_2\in (0,1]$ define 
\begin{equation}
\begin{split}
C^{(G)}\big((x_1,f_1), (x_2,f_2)\big):= &-\frac{1}{\pi^2}\int_{\DD_{x_1}}\!\!\!\!\dif^2 z_1\int_{\DD_{x_2}}\!\!\!\!\dif^2 z_2 \, \overline{\partial_{z_1} f_1(z_1)}\partial_{z_2} f_2(z_2) \partial_{z_1}\partial_{\overline{z}_2} \log K\big((x_1,z_1),(x_2,z_2)\big)\\
&+\frac{\min\{x_1,x_2\}}{2\pi^2}\int_{\partial \DD_{x_1}\times \partial \DD_{x_2}} \frac{\overline{f_1(z_1)}f_2(z_2)}{\left(\min\{x_1,x_2\}-z_1\overline{z}_2\right)^2}\dif z_1\wedge\dif \overline{z}_2.
\end{split}
\label{eq:C_Gauss}
\end{equation}
Here $\DD_r:=\sqrt{r}\, \DD$ for $r>0$, and
 $K$ is given by
\begin{equation}
K\big((x_1,z_1),(x_2,z_2)\big):= |z_1-z_2|^2 + |x_1-x_2|B,\quad B=B\big((x_1,z_1),(x_2,z_2)\big):=\begin{cases}
1-\frac{|z_1|^2}{x_1},& x_1\ge x_2,\\
1-\frac{|z_2|^2}{x_2},&x_1<x_2.
\end{cases}
\label{eq:def_K}
\end{equation}
Next, we set
\begin{equation}
\begin{split}
&C^{(4)}\big((x_1,f_1), (x_2,f_2)\big):=\\
&\frac{\min\{x_1,x_2\}}{\max\{x_1,x_2\}}\!\!\left(\!\frac{1}{\pi x_1}\int_{\DD_{x_1}}\!\!\!\!\overline{f_1(z)}\dif^2 z-\frac{1}{2\pi}\int_0^{2\pi}\!\!\overline{f_1(\sqrt{x_1}\ee^{\ii\theta})}\dif\theta\!\right)\!\!\!\left(\!\frac{1}{\pi x_2}\int_{\DD_{x_2}}\!\!\!\!f_2(z)\dif^2 z-\frac{1}{2\pi}\int_0^{2\pi}\!\!f_2(\sqrt{x_2}\ee^{\ii\theta})\dif\theta\!\right)\!.
\end{split}
\label{eq:C_4}
\end{equation}
Finally, we define the bilinear forms $C_\beta$, $\beta=1,2$, determining the behavior of \eqref{eq:Cov_intro} in the real ($\beta=1$) and complex ($\beta=2$) case. Denote
\begin{equation}
\kappa_4:=\E|\chi|^4-1-2/\beta.
\label{eq:def_kappa_4}
\end{equation}
For $\beta=2$ we set 
\begin{equation}
C_{\beta=2}\big((x_1,f_1),(x_2,f_2)\big) := C^{(G)}\big((x_1,f_1),(x_2,f_2)\big) + \kappa_4 C^{(4)}\big((x_1,f_1),(x_2,f_2)\big).
\label{eq:C_beta_2}
\end{equation}
For $\beta=1$ we further define
\begin{equation}
C_{\beta=1}\big((x_1,f_1),(x_2,f_2)\big) := 2C^{(G)}\big((x_1,P_\mathrm{sym}f_1),(x_2,P_\mathrm{sym}f_2)\big) + \kappa_4 C^{(4)}\big((x_1,f_1),(x_2,f_2)\big),
\label{eq:C_beta_1}
\end{equation}
where the symmetrization operator is given by
\begin{equation}
\left(P_{\mathrm{sym}}f\right)(z):=\frac{f(z)+f(\overline{z})}{2},\quad \forall z\in\C.
\label{eq:P_sym}
\end{equation}

The superscript in $C^{(G)}$ refers to the fact that the corresponding terms in \eqref{eq:C_beta_2} and \eqref{eq:C_beta_2} are the only remaining ones when $X$ is a Ginibre matrix, i.e. when the entries of $X$ have Gaussian distribution. The quadratic form $C^{(4)}$ is a correction to the Gaussian result and appears in \eqref{eq:C_beta_2}, \eqref{eq:C_beta_1} multiplied by the fourth cumulant of the single-entry distribution. The appearance of symmetrization $P_{\mathrm{sym}}$ in the real case is due to the fact that the spectrum of $X$ (and of its minors) is symmetric with respect to the real axis.

All integrals in \eqref{eq:C_Gauss} are well-defined, we postpone a detailed discussion of this topic to Appendix~\ref{sec:C_properties}.

\subsection{Mesoscopic CLT}\label{sec:mesoCLT}

In the mesoscopic regime, we fix $a\in (0,1/2)$, $x_0\in(0,1)$, $|z_0|<\sqrt{x_0}$, all $N$-independent, and set
\begin{equation}
f_{j,z_0,a}(z):=f_j\left(N^a(z-z_0)\right),\quad z\in\C,\qquad x_j:=x_0+N^{-2a}y_j,\quad y_j\in\R,
\label{eq:meso_scaling}
\end{equation} 
for $j\in[n]$. As Theorem~\ref{theo:main_meso} below shows, the limiting behavior of covariance \eqref{eq:Cov_intro} in this regime is simpler compared to the macroscopic regime. The contribution from $\kappa_4$ is dominated by the Gaussian contribution, and the limit becomes universal; confirming the natural intuition that spectral statistics on smaller scales are "more universal". Moreover, since we rescale test functions around some point in the bulk, the edge contribution in the second line of \eqref{eq:C_Gauss} disappears. In such a way, we are left only with the analogue of the first term in the rhs. of \eqref{eq:C_Gauss}. Define the local version of the kernel $K$ by
\begin{equation}
K_{x_0,z_0}\big((y_1,z_1),(y_2,z_2)\big):= |z_1-z_2|^2 + |y_1-y_2| \left(1-\frac{|z_0|^2}{x_0}\right).
\end{equation}
The mesoscopic analogue of $C^{(G)}$ is given by
\begin{equation}
C^{(G)}_{x_0,z_0}\big((y_1,f_1),\!(y_2,f_2)\big)\!:=\! -\frac{1}{\pi^2}\!\int_{\DD_{x_0}}\!\!\!\!\!\dif^2 z_1\!\!\int_{\DD_{x_0}}\!\!\!\!\!\dif^2 z_2 \, \overline{\partial_{z_1} f_1(z_1)}\partial_{z_2} f_2(z_2) \partial_{z_1}\partial_{\overline{z}_2} \log K_{x_0,z_0}\big((y_1,z_1),\!(y_2,z_2)\big)\!.
\end{equation}
We also define 
\begin{equation}
C_{x_0,z_0,\beta=2}:=C_{x_0,z_0}^{(G)},\quad C_{x_0,z_0,\beta=1}\big((y_1,f_1),(y_2,f_2)\big):=2C_{x_0,z_0}^{(G)}\big((y_1,P_{\mathrm{sym}}f_1), (y_2,P_{\mathrm{sym}}f_2)\big).
\label{eq:C_beta_meso}
\end{equation}
We point out that an analogous kernel previously appeared in the context of dynamics on non-Hermitian matrices in \cite[Eq.(2.7)]{bourgade2024fluctuations}. Our second main result is as follows.

\begin{theorem}[Mesoscopic CLT]\label{theo:main_meso} Let $X$ be a complex or real $N\times N$ i.i.d. matrix satisfying Assumption~\ref{ass:chi}. Fix (small) $c_0, \tau>0$, an exponent $a\in (0,1/2)$ and an open bounded domain $\Omega\subset \C$. For $x_0\in [c_0,1)$ and $z_0\in\C$ with $|z_0|\le (1-\tau)\sqrt{x_0}$, there exists the centered Gaussian process $L_{x_0,z_0}(y,f)$ indexed by $\R\times H^2_0(\Omega)$ with the covariance structure given by $C_{x_0,z_0,\beta}$ defined in \eqref{eq:C_beta_meso}. The process of mesoscopically scaled centered linear statistics $L^{\lceil xN\rceil}(f_{z_0,a})$ with $x:=x_0+N^{-2a}y$ converges to $L_{x_0,z_0}(y,f)$ in the sense of moments \eqref{eq:main_macro}. The implicit constant in the error term may depend only on the moment degree, $c_0,\tau,a$, model parameters from Assumption~\ref{ass:chi} and $H^2(\Omega)$ norms of test functions.
\end{theorem}

Similarly to Corollary~\ref{cor:general_testfn}, Theorem~\ref{theo:main_meso} can be also extended to test functions with 3-dimensional support.

\section{Proof strategy}\label{sec:strategy}

The proof of Theorem~\ref{theo:main} closely follows the strategy introduced in \cite{macroCLT_real, macroCLT_complex}, so for brevity we focus on the proof of \eqref{eq:main_macro} for the covariance of $L^{(k_1)}(f_1)$ and $L^{(k_2)}(f_2)$. Specifically, we show that for any fixed $f_1,f_2\in H^{2+\delta}_0(\Omega)$ and for any (possibly, $N$-dependent) $x_j:=k_j/N\in [c_0,1]$, $j=1,2$, it holds that
\begin{equation}
\Cov\left(L^{(k_1)}(f_1), L^{(k_2)}(f_2)\right) = C_\beta \left((x_1,f_1),(x_2,f_2)\right) + \mathcal{O}(N^{-c}),
\label{eq:Cov_aim}
\end{equation}
for some $N$-independent $c>0$. The derivation of \eqref{eq:main_macro} for higher order moments from the inputs presented below is standard and thus is postponed to Appendix~\ref{sec:Wick}. The proof of Theorem~\ref{theo:main_meso} additionally requires only the inputs already available from \cite{Cipolloni_meso}, for more details see Section~\ref{sec:CLTproofs}.

The proof of \eqref{eq:Cov_aim} consists of two steps outlined in Sections~\ref{sec:strategy_Girko} and~\ref{sec:strategy_Cov}. First, we apply the Girko's Hermitization formula and remove the singular regimes from it. This is the content of Lemma~\ref{lem:Girko_reduction}, which is borrowed from \cite{macroCLT_complex}. In the second step we compute the covariance of two resolvent traces in Proposition~\ref{prop:underline_comp}, which is a new technical input.

\subsection{The Girko's Hermitization formula}\label{sec:strategy_Girko} As usual,  for $k\in[N]$ we set $x:=k/N$. For $z\in \C$ and $w\in\C\setminus\R$ denote
\begin{equation}
W^{(k)}:=\begin{pmatrix}
0&X^{(k)}\\\big(X^{(k)}\big)^*&0
\end{pmatrix},\,\,
Z^{(k)}:=\begin{pmatrix}
0&z\\
\overline{z}&0
\end{pmatrix},\,\,H^{(k),z}:=W^{(k)}-Z^{(k)},\,\, G^{(k),z}_x(w):=\big( H^{(k),z}-w\big)^{-1}. 
\label{eq:def_hermitization}
\end{equation}
All matrices defined in \eqref{eq:def_hermitization} have size $(2k)\times (2k)$. In $G^{(k),z}_x$ the subscript indicates that the entries of $H^{(k),z}$ have variance $N^{-1/2} = x^{1/2}k^{-1/2}$ instead of the standard normalization $k^{-1/2}$. The same subscript could be included into the notations $W^{(k)}$ and $H^{(k),z}$ for the same reason, however we avoid this for brevity. We call $W^{(k)}$ the \emph{Hermitization} of $X^{(k)}$, $z$ the \emph{Hermitization parameter}, and $w$ the \emph{spectral parameter}. The Girko's Hermitization formula \cite{Girko84} in the regularized form given in \cite{TaoVu15} states that
\begin{equation}
\sum_{i=1}^k f(\sigma_i^{(k)})=\frac{1}{4\pi} \int_\C \Delta f(z) \log\left\vert \det (H^{(k),z}-\ii T)\right\vert \dif^2 z - \frac{k}{2\pi\ii} \int_\C \Delta f(z)\dif^2 z \int_0^T \big\langle G^{(k),z}_x(\ii\eta)\big\rangle\dif \eta,
\label{eq:Girko}
\end{equation}
where we choose $T:=N^{100}$. For a small fixed $\delta_c>0$ which will be chosen later we set
\begin{equation}
I_{c}^{(k)}(f):=- \frac{k}{2\pi\ii} \int_{\C} \Delta f(z)\dif^2 z \int_{\eta_c}^{T} \big\langle G^{(k),z}_x(\ii\eta)\big\rangle\dif \eta,\quad \eta_c:=N^{-1+\delta_c}.
\label{eq:def_I_c}
\end{equation}
Here we removed the small $\eta$ regime from consideration that is shown to be negligible. Indeed, the following statement shows that the leading contribution to the lhs. of \eqref{eq:Cov_aim} comes from $I_c^{(k_j)}(f_j)$.

\begin{lemma}\label{lem:Girko_reduction} Assume the set-up and conditions of Theorem~\ref{theo:main}. There exists a (small) fixed $\delta_c>0$ such that
\begin{equation}
\mathrm{Cov}\left(L^{(k_1)}(f_1), L^{(k_2)}(f_2)\right) = \mathrm{Cov}\left(I_{c}^{(k_1)}(f_1), I_{c}^{(k_2)}(f_2)\right) + \mathcal{O}(N^{-c}),
\label{eq:Girko_reduction}
\end{equation}
for any $f_j\in H^{2+\delta}_0(\Omega)$ and $k_j\in[N]$ with $x_j:=k_j/N\in [c_0,1]$, $j=1,2$. The constant $c>0$ does not depend on $N$, and the implicit constant in $\mathcal{O}(\cdot)$ may depend on $\|\Delta f_j\|_{L^2(\Omega)}$ and $|\Omega|$. 
\end{lemma}

The proof of Lemma~\ref{lem:Girko_reduction} immediately follows from \cite[Lemma~4.1]{macroCLT_complex} for $\beta=2$ and \cite[Proposition~3.4]{macroCLT_real} for $\beta=1$, and thus is omitted.

\subsection{Covariance of two resolvent traces}\label{sec:strategy_Cov} Lemma~\ref{lem:Girko_reduction} reduces the calculation of the lhs. of \eqref{eq:Girko_reduction} to the calculation of covariance of two resolvent traces. The result of this calculation is presented in Proposition~\ref{prop:underline_comp} below. Let us introduce the framework for this result. The \emph{self-energy operator} $\mathcal{S}^{(k)}_x$ is given by
\begin{equation}
\begin{split}
&\mathcal{S}^{(k)}_x[R] := x\left(\big\langle RE_+^{(k)}\big\rangle E_+^{(k)} -\big\langle RE_-^{(k)}\big\rangle E_-^{(k)}\right),\quad E_\pm^{(k)}:=\begin{pmatrix}
1&0\\0&\pm 1
\end{pmatrix},\quad \forall R\in \C^{(2k)\times (2k)}.
\end{split}
\label{eq:def_S}
\end{equation}
The blocks of $E^{(k)}_\pm$ are of size $k\times k$. For $\beta=2$ the rhs. of \eqref{eq:def_S} actually coincides with $\E\big[ W^{(k)}RW^{(k)}\big]$. For $\beta=1$ this expectation slightly differs from the rhs. of \eqref{eq:def_S}, however we still use \eqref{eq:def_S} as the definition of $\mathcal{S}_x^{(k)}$ for technical convenience. For $z\in\C$ and $w\in\C\setminus\R$, the \emph{matrix Dyson equation} (MDE) is defined by
\begin{equation}
-M^{-1}=w-Z^{(k)}+\mathcal{S}^{(k)}_x[M],\quad \Im w\Im M>0,\quad M\in\C^{(2k)\times (2k)}.
\label{eq:MDE}
\end{equation}
It is known from \cite[Lemma~2.2]{loc_circ_law18} that the solution to \eqref{eq:MDE} exists and is unique (see also the abstract result in \cite{Helton07}). We denote this solution by $M:=M^{(k),z}_x(w)$. Typically, the equation \eqref{eq:MDE} is considered in the literature for $x=1$, while the general case $x\in (0,1]$ can be reduced to $x=1$ by rescaling:
\begin{equation}
M_x^{(k),z}(w) =x^{-1/2} M_1^{(k),x^{-1/2}z}\left(x^{-1/2}w\right),\quad G_x^{(k),z}(w) =x^{-1/2} G_1^{(k),x^{-1/2}z}\left(x^{-1/2}w\right).
\label{eq:M_rescaling}
\end{equation}
These relations allow us to transfer results stated for $x=1$ to the general case $x\in(0,1]$, which we will implicitly use from now on. In particular, \eqref{eq:M_rescaling}, in combination with \cite[Eq.(3.5)-(3.6)]{loc_circ_law18}, implies that $M$ has the following $2\times 2$ block-constant structure with the blocks of size $k\times k$:
\begin{equation}
M=\begin{pmatrix}
m&-zu\\ -\overline{z}u&m
\end{pmatrix},\quad u=u^{z}_x(w):=\frac{m}{w+xm},
\label{eq:def_u}
\end{equation}
where $m=m^{z}_x(w)$ solves the cubic equation
\begin{equation}
-\frac{1}{m} = w+xm -\frac{|z|^2}{w+xm},\quad \Im w\Im m>0.
\label{eq:m_cubic}
\end{equation}
By \cite[Eq.(3.3)]{macroCLT_complex} it holds that
\begin{equation}
\|M\|+|m|+|u|\lesssim x^{-1}.
\label{eq:Mmu_bound}
\end{equation}
Finally, we denote
\begin{equation}
\rho=\rho^z_x(w):=\frac{1}{\pi}|\Im m^z_x(w)|.
\label{eq:def_rho}
\end{equation}

The solution to \eqref{eq:MDE} is called the \emph{deterministic approximation} to $G^{(k),z}_x$, since this resolvent concentrates around $M^{(k),z}_x$ in the sense of the following \emph{single-resolvent local laws} proven earlier.
\begin{theorem} \label{theo:1Gllaw}
Let $X$ satisfy Assumption~\ref{ass:chi}(i). Fix a (large) $L>0$ and a (small) $c_0>0$. There exists an $N$-independent constant $\tau>0$ such that
\begin{align}
\left\vert\left\langle \big(G^{(k),z}_x(w)-M^{(k),z}_x(w)\big)B\right\rangle\right\vert &\prec\frac{\|B\|}{N\eta},\label{eq:1G_av}\\
\left\vert\left\langle \bm{x}, \big(G^{(k),z}_x(w)-M^{(k),z}_x(w)\big)\bm{y}\right\rangle\right\vert &\prec\|\bm{x}\|\|\bm{y}\|\left(\sqrt{\frac{\rho_x^{z}(w)}{N\eta}}+\frac{1}{N\eta}\right),\label{eq:1G_iso}
\end{align}
uniformly in deterministic matrix $B\in\C^{(2k)\times (2k)}$, deterministic vectors $\bm{x},\bm{y}\in\C^{2k}$, $x:=k/N\in [c_0,1]$, $|\Re w|\le \tau$, $\eta:=|\Im w|\in (N^{-L}, N^L)$ and $|z|\le L$.
\end{theorem}
For a sufficiently small fixed $\theta>0$, the regime $|\sqrt{x}-|z||\ge \theta$ in Theorem~\ref{theo:1Gllaw} is covered by \cite[Theorem~3.1]{macroCLT_real} and the complementary edge regime $|\sqrt{x}-|z||\le \theta$ by \cite[Theorem~3.1]{spectral_radius}.

Now we are ready to state the main technical result of Section~\ref{sec:strategy_Cov}.

\begin{proposition}\label{prop:underline_comp} Let $X$ satisfy Assumption~\ref{ass:chi}(i). Fix a (large) $L>0$ and (small) $c_0,\epsilon,\xi>0$. Let $\theta>0$ be a (possibly, $N$-dependent) small parameter. Consider $k_j\in [N]$, $z_j\in\C$, and $w_j=\ii\eta_j$ with $\eta_j>0$, for $j=1,2$. Denote $x_j:=k_j/N$, $G_j:=G^{(k_j),z_j}_{x_j}(w_j)$ and assume that $x_1\le x_2$. Recall the definition of $\kappa_4$ from \eqref{eq:def_kappa_4}. It holds that
\begin{equation}
\E \langle G_1-\E G_1\rangle \langle G_2-\E G_2\rangle = \frac{1}{2x_1x_2N^2}\left(V_{12}+\bm{1}_{\beta=1}V_{1\overline{2}}+\kappa_4 x_1^2 U_1U_2\right) + \mathcal{O}_\theta\left(\left(\frac{1}{N\eta_*}+\frac{1}{\sqrt{N}}\right)\frac{N^\xi}{N^2\eta_1\eta_2}\right),
\label{eq:Cov_UV}
\end{equation}
uniformly in $x_j\in [c_0,1]$, $z_j\in\C$ with $\theta\le |\sqrt{x_j}-|z_j||\le L$, and $N^{-1+\epsilon}\le \eta_j\le N^L$, for $j=1,2$. In \eqref{eq:Cov_UV} we denoted $\eta_*:=\min\{\eta_1,\eta_2\}$ and
\begin{equation}
\begin{split}
U_j=U_j(x_j,\eta_j,z_j):=&\frac{\ii}{\sqrt{2}}\partial_{\eta_j}m_j^2,\quad j=1,2,\\
V_{12}:=V\big((x_1,\eta_1,z_1),(x_2,\eta_2,&z_2)\big),\quad V_{1\overline{2}}:=V\big((x_1,\eta_1,z_1),(x_2,\eta_2,\overline{z}_2)\big),\\ 
V\big((x_1,\eta_1,z_1),(x_2,\eta_2,z_2)\big):=&\frac{1}{2}\partial_{\eta_1}\partial_{\eta_2}\log \left[1+x_1^2\left(u_1u_2|z_1z_2|\right)^2 -x_1^2m_1^2m_2^2 -2x_1u_1u_2\Re [z_1\overline{z}_2]\right],\\
\end{split}
\label{eq:def_UV}
\end{equation}
where $m_j:=m^{z_j}_{x_j}(w_j)$ and $u_j:=u^{z_j}_{x_j}(w_j)$.
The implicit constant in $\mathcal{O}_\theta(\cdot)$ is bounded from above by $\theta^{-D}$ for some (large) fixed $D>0$. 
\end{proposition}

In the special case $x_1=x_2=1$, the leading term in the rhs. of \eqref{eq:Cov_UV} was first obtained in \cite[Theorem~3.3]{macroCLT_complex} for $\beta=2$ and in \cite[Theorem~3.3]{macroCLT_real} for $\beta=1$. The bound on the error term in the rhs. of \eqref{eq:Cov_UV} was then subsequently improved in \cite[Proposition~3.4]{Cipolloni_meso} and \cite[Proposition~3.4]{hyperuniformity}. In the current paper we do not aim to get an optimal bound on the error term, since this is not needed for the proof of Theorem~\ref{theo:main}. Observe that the local law \eqref{eq:1G_av} implies the bound of order $(N^2\eta_1\eta_2)^{-1}$ on the lhs. of \eqref{eq:Cov_UV} (up to the factor $N^\xi$), which is indeed the size\footnote{It also holds that the leading term in the rhs. of \eqref{eq:Cov_UV} has an upper bound of order $C_\theta(N^2\eta_1\eta_2)^{-1}$ in absolute value, with some constant $C_\theta\lesssim \theta^{-D}$ for sufficiently large $D>0$. See e.g. \eqref{eq:abs_V_int1} and Lemma~\ref{lem:beta} for the bound on $|V_{12}|$, and  \eqref{eq:A_bound}, \eqref{eq:MA} for the bound on $|U_j|$.}  of the leading term in the rhs. of \eqref{eq:Cov_UV} in the case when $x_1=x_2$, $z_1=z_2$ and $w_1=w_2$. Using the chaos expansion technique from \cite[Section~7]{hyperuniformity} we could also show that the first term is the dominating one for general $x_j,z_j,w_j$. This could be done by incorporating $|z_1-z_2|$ dependence into the bound on the error term and getting that this bound improves as $|z_1-z_2|$ increases. However we do not pursue this line of research here.

The key ingredient in the proof of Proposition~\ref{prop:underline_comp} is the concentration bound for the product of $G_1$ and $G_2$. Results of this type are commonly referred to as two-resolvent local laws. Since the sizes of $G_1$ and $G_2$ do not match, we need to project $G_2$ on the space of $(2k_1)\times (2k_1)$ matrices. To this end, denote
\begin{equation}
\label{eq:defP}
\widetilde{G}_2:=P^*G_2P,\quad\text{where}\quad P:=\begin{pmatrix}
0&0\\1&0\\0&0\\0&1
\end{pmatrix}\in \C^{(2k_2)\times (2k_1)}.
\end{equation}
The matrix $P$ is split into columns of width $k_1,k_1$ and the rows of height $k_2-k_1, k_1,k_2-k_1,k_1$. It is defined in such a way that
\begin{equation}
W_1=P^*W_2P,\quad \text{where}\quad W_j:=W^{(k_j)},\,\,j=1,2. 
\end{equation}
We show that the product of $G_1$ and $\widetilde{G}_2$ concentrates around its deterministic approximation, which is the content of Proposition~\ref{prop:2G_randomD} below.

Let $B\in\C^{(2k_1)\times (2k_1)}$ be a deterministic matrix, which we customarily call an \emph{observable}. The deterministic approximation to $G_1B\widetilde{G}_2$ is constructed from the deterministic approximations to $G_1$ and $\widetilde{G}_2$, as we now explain. First, denote
\begin{equation}
M_j:=M^{(k_j),z_j}_{x_j}(w_j),\, \, j=1,2,\quad\text{and}\quad \widetilde{M}_2:=P^*M_2P.
\label{eq:M_tilde}
\end{equation}
Note that $\widetilde{M}_2$ differs from $M_2$ only by shrinking the size of blocks from $k_2\times k_2$ to $k_1\times k_1$, while the constants inside the blocks remain intact. The \emph{two-body stability operator} associated to $G_1$, $\widetilde{G}_2$ is given by
\begin{equation}
\mathcal{B}_{1\widetilde{2}}[R]:=R - M_1\mathcal{S}^{(k_1)}_{x_1}[R]\widetilde{M}_2,\quad \forall R\in\C^{(2k_1)\times (2k_1)}.
\label{eq:def_B_12}
\end{equation}
Finally, the deterministic approximation to $G_1B\widetilde{G}_2$ is defined by
\begin{equation}
\label{eq:M2def}
M_{1\widetilde{2}}^B:=\mathcal{B}_{1\widetilde{2}}^{-1}[M_1B\widetilde{M}_2]\in\C^{(2k_1)\times (2k_1)}.
\end{equation}
For further reference, we also define the \emph{one-body stability operator} associated to $G_j$ and the deterministic approximation to $G_jBG_j$:
\begin{align}
\mathcal{B}_{jj}[R]:=&R - M_j\mathcal{S}^{(k_j)}_{x_j}[R]M_j,\quad \forall R\in\C^{(2k_j)\times (2k_j)},\label{eq:def_B_11}\\
M_{jj}^B:=&\mathcal{B}_{jj}^{-1}[M_jBM_j],\quad B\in \C^{(2k_j)\times (2k_j)},\,\, j=1,2.\label{eq:def_M_11}
\end{align}

We state the two-resolvent local law for the averaged quantities of the form $\langle G_1B_1\widetilde{G_2}B_2\rangle$ in a narrow cone containing the imaginary axis. This means that the spectral parameters $w_1$, $w_2$ of $G_1$ and $\widetilde{G}_2$ are not necessarily purely imaginary, but satisfy $|\Re w_j|\ll |\Im w_j|$ in a sense explained in Proposition~\ref{prop:2G_randomD}. As at least one of the parameters $w_j$ approaches the boundary of $\DD_{x_j}$, the cone becomes more narrow. In fact, the restriction of $w_1,w_2$ to a cone could be removed, however this is not needed for the proof of Proposition~\ref{prop:underline_comp}, so we keep this restriction for simplicity.

\begin{proposition}[Two-resolvent averaged local law]\label{prop:2G_randomD} Let $X$ satisfy Assumption~\ref{ass:chi}(i).  Fix a (large) $L>0$ and (small) $c_0,\epsilon,\nu,\xi>0$. Let $\theta>0$ be a (possibly, $N$-dependent) small parameter.  Consider $k_j\in[N]$ and denote $x_j:=k_j/N$, $j=1,2$. There exists a universal constant $d>0$ such that
\begin{equation}
\left\vert\left\langle \left(G_1B_1\widetilde{G}_2 - M_{1\widetilde{2}}^{B_1}\right)B_2\right\rangle\right\vert \le C_\theta \frac{N^\xi}{N\eta_1\eta_2},
\label{eq:2G_av}
\end{equation}
uniformly in $\eta_j:=|\Im w_j|\in (N^{-1+\epsilon}, N^L)$, $|\Re w_j|\le \theta^d N^{-\nu}\eta_j$, $z_j\in \C$ such that $\theta\le|z_j-\sqrt{x_j}|\le L$, $x_j\in [c_0,1]$ and in deterministic matrices $B_j\in \{E_\pm^{(k_j)},\big(F^{(k_j)}\big)^{(*)}\}$, $j=1,2$. The constant $C_\theta>0$ in \eqref{eq:2G_av} satisfies $C_\theta\le \theta^{-D}$ for some fixed $D>0$.
\end{proposition}

The proof of Proposition~\ref{prop:2G_randomD} is presented in Section~\ref{sec:local_law_proof}. It relies on the \emph{zigzag} strategy introduced in \cite{cipolloni2024out}, see also \cite{Cipolloni_meso, cipolloni2023eigenstate}. One of the main ingredients in this proof is the analysis of the eigenvalue of $\mathcal{B}_{1\widetilde{2}}$ with the least absolute value, see Lemma~\ref{lem:beta} below. 

\section{CLT for minors: Proof of Theorems~\ref{theo:main} and \ref{theo:main_meso}}
\label{sec:CLTproofs}

In this section we prove Theorem~\ref{theo:main} relying on the technical results listed in Section~\ref{sec:strategy}. We focus on the proof of \eqref{eq:Cov_aim}, while for higher order moments the argument is standard and thus is postponed to Appendix~\ref{sec:Wick}. First, we compute the lhs. of \eqref{eq:Cov_aim} to the leading order in terms of $V_{12}$ and $U_j$ introduced in~\eqref{eq:def_UV}.

\begin{lemma}\label{lem:determ_truncation} Assume the set-up and conditions of Theorem~\ref{theo:main}. For $k_j\in[N]$ denote $x_j:=k_j/N$, $j=1,2$, and assume that $x_1\le x_2$. For any $f_1,f_2\in H^{2+\delta}_0(\Omega)$ it holds that
\begin{equation}
\Cov \left(L^{(k_1)}(f_1), L^{(k_2)}(f_2)\right) = \frac{1}{8\pi^2}\left(I_1+\bm{1}_{\beta=1}\widetilde{I}_1 +\kappa_4 x_1^2 I_2\right)+\mathcal{O}\left(N^{-c}\right),
\label{eq:determ_truncation}
\end{equation}
for some $N$-independent $c>0$. Here we denoted
\begin{equation}
\begin{split}
&\int^\#: =   \int_\C\!\!\dif^2 z_1\!\!\int_\C\!\!\dif^2 z_2 \overline{\Delta f_1(z_1)}\Delta f_2(z_2)\int_0^{+\infty}\!\!\!\dif\eta_1 \!\!\int_0^{+\infty}\dif\eta_2,\\
I_1&: = \int^\#(-V_{12}), \qquad \widetilde{I}_1:= \int^\#(-V_{1\bar 2}),\qquad I_2:=\int^\# (-U_1U_2).
\end{split}
\end{equation}
The implicit constant in $\mathcal{O}(\cdot)$ in the rhs. of \eqref{eq:determ_truncation} is uniform in $x_j\in [c_0,1]$, $j=1,2$, and may depend on $\|\Delta f_j\|_{L^2(\Omega)}$ and $|\Omega|$.
\end{lemma}

The proof of Lemma~\ref{lem:determ_truncation} is presented in Appendix~\ref{sec:truncation}. Next, we explicitly compute the integrals in the rhs. of \eqref{eq:determ_truncation}.

\begin{lemma}\label{lem:UV_int} Let the test functions $f_1,f_2$ satisfy the assumptions of Theorem~\ref{theo:main} and let $K$ be defined as in \eqref{eq:def_K}. For any $x_1,x_2\in (0,1]$ with $x_1<x_2$ it holds that
\begin{equation}
I_1= -8\int_{\DD_{x_1}}\dif^2 z_1\int_{\DD_{x_2}}\dif^2 z_2\overline{\partial_{z_1}f_1(z_1)}\partial_{z_2}f_2(z_2)\partial_{z_1}\partial_{\overline{z}_2}\log K +4x_1\!\!\int_{\partial\DD_{x_1}\times \partial \DD_{x_2}} \frac{\overline{f_1(z_1)}f_2(z_2)}{(x_1-z_1\overline{z}_2)^2}\dif z_1\wedge \dif\overline{z}_2,\label{eq:V_int}
\end{equation}
\begin{equation}
I_2=\frac{8}{x_1x_2}\left(\frac{1}{x_1}\int_{\DD_{x_1}}\overline{f_1(z)}\dif^2 z-\frac{1}{2}\int_0^{2\pi}\overline{f_1(\sqrt{x_1}\ee^{\ii\theta})}\dif\theta\right)\!\!\left(\frac{1}{x_2}\int_{\DD_{x_2}}f_2(z)\dif^2 z-\frac{1}{2}\int_0^{2\pi}f_2(\sqrt{x_2}\ee^{\ii\theta})\dif\theta\right)\!.\label{eq:U_int}
\end{equation}
\end{lemma}

The proof of Lemma~\ref{lem:UV_int} is postponed to the end of this section. Now we are ready to complete the proof of Theorem~\ref{theo:main} and of its mesoscopic version Theorem~\ref{theo:main_meso}.

\begin{proof}[Proof of Theorem~\ref{theo:main} for covariance] We start with the proof of \eqref{eq:Cov_aim} in the complex case ($\beta=2$). Combining Lemma~\ref{lem:determ_truncation} with Lemma~\ref{lem:UV_int}, we immediately conclude \eqref{eq:Cov_aim} for $x_1<x_2$. Since $C_{\beta=2}$ is sesquilinear in arguments $(x_1,f_1)$ and $(x_2,f_2)$, we automatically get \eqref{eq:Cov_aim} for $x_1>x_2$. In the special case $x_1=x_2=x$ we also use Lemma~\ref{lem:determ_truncation}, however the calculation in Lemma~\ref{lem:UV_int} is slightly different. For more details see \eqref{eq:C_limit} below.

In the real case ($\beta=1$) the argument is analogous. We only need to accommodate the term $V_{1\overline{2}}$ in the rhs. of \eqref{eq:determ_truncation}, which we do by observing that
\begin{equation}
C^{(G)}\big((x_1,f_1),(x_2,f_2)\big) = C^{(G)}\big((x_1,\widetilde{f}_1),(x_2,\widetilde{f}_2)\big),\quad \widetilde{f}_j(z):=f_j(\overline{z}),
\end{equation}
for any $f_1,f_2\in H^{2+\delta}(\Omega)$. This identity follows from the change of variables $z_j\to \overline{z}_j$ in \eqref{eq:C_Gauss} (see also the symmetric way of writing $C^{(G)}$ in \eqref{eq:C_Gauss_app}), and gives rise to the first term in the rhs. of \eqref{eq:C_beta_1}.
\end{proof}

\begin{proof}[Proof of Theorem~\ref{theo:main_meso}] The proof is almost identical to that of Theorem~\ref{theo:main}. The computation of the covariance is analogous to the macroscopic case above. Then, the CLT follows by a Wick's Theorem as in Appendix~\ref{sec:Wick}. In fact, the proof in Appendix~\ref{sec:Wick} operates only with resolvents, hence it can directly be used in Girko's formula both in the macroscopic and in the mesoscopic cases. We note that the proof of the mesoscopic analogue of Lemma~\ref{lem:determ_truncation} requires a more precise control on the regime $0<\eta\le \eta_c$ in the Girko's formula, which is though available from \cite[Proposition~3.5]{Cipolloni_meso}.
\end{proof}

We conclude this section with the proof of Lemma~\ref{lem:UV_int}.

\begin{proof}[Proof of Lemma~\ref{lem:UV_int}] The identity \eqref{eq:U_int} follows from the analogous result for $x_1=x_2=1$ from \cite[Lemma~4.10]{macroCLT_complex} via rescaling \eqref{eq:M_rescaling} and thus is omitted. We prove \eqref{eq:V_int} by integrating first the third line of \eqref{eq:def_UV} over $\eta_1,\eta_2$:
\begin{small}
\begin{equation}
\int_0^{+\infty}\!\!\!\dif\eta_1\int_0^{+\infty}\!\!\!\dif\eta_2 V_{12} = \frac{1}{2}\lim_{\eta_1,\eta_2\to +0} \log \left[1+x_1^2\left(u_1u_2|z_1z_2|\right)^2 -x_1^2m_1^2m_2^2 -2x_1u_1u_2\Re [z_1\overline{z}_2]\right]=:-\Theta(z_1,z_2).
\label{eq:V_to_Theta}
\end{equation}
\end{small}
To compute this limit, observe that by \eqref{eq:m_cubic} it holds that
\begin{equation}
\lim_{\eta_j\to+0} m_j = \frac{\ii}{x_j}\sqrt{x_j-|z_j|^2}\,\bm{1}_{|z_j|^2 \le x_j},\qquad \lim_{\eta_j\to +0}u_j =\frac{1}{x_j}\bm{1}_{|z_j|^2\le x_j} + \frac{1}{|z_j|^2}\bm{1}_{|z_j|^2> x_j}.
\label{eq:u_at_origin}
\end{equation}
For analogous limits in the case $x_j=1$ we refer to \cite[Sec.4.3.2-4.3.3]{macroCLT_complex}. Combining \eqref{eq:V_to_Theta} with \eqref{eq:u_at_origin} we get
\begin{equation}
\Theta(z_1,z_2)=\frac{1}{2}
\begin{cases}
-\log K\big((x_1,z_1),(x_2,z_2)\big) + \log x_2,& z_1\in\DD_{x_1},\, z_2\in \DD_{x_2},\\
\log |z_1|^2 -\log |x_2z_1-x_1z_2|^2+\log x_2^2,& z_1\in\DD_{x_1}^c,\, z_2\in \DD_{x_2},\\
\log |z_2|^2 -\log|z_1-z_2|^2,& z_1\in\DD_{x_1},\, z_2\in \DD_{x_2}^c,\\
\log |z_1z_2|^2 -\log |x_1-z_1\overline{z}_2|^2,& z_1\in\DD_{x_1}^c,\, z_2\in \DD_{x_2}^c.
\end{cases}
\label{eq:def_Theta}
\end{equation}
Here we recall the definition of $K$ from \eqref{eq:def_K} and denoted $\DD_{x_j}^c:=\C\setminus\DD_{x_j}$ for $j=1,2$. Note that $\Theta$ is not singular: the arguments of logarithms in the rhs. of \eqref{eq:def_Theta} do not vanish for any $z_1,z_2\in\C$.

By \eqref{eq:V_to_Theta}, the lhs. of \eqref{eq:V_int} equals to
\begin{equation}
\int_\C \dif^2 z_1 \int_\C \dif^2 z_2 \overline{\Delta f_1(z_1)}\Delta f_2(z_2)\Theta(z_1,z_2).
\label{eq:Theta_int}
\end{equation}
We compute \eqref{eq:Theta_int} by integration by parts. First, we split the region of integration into four regimes $\Omega_1\times \Omega_2$ with $\Omega_j\in \{\DD_{x_j}, \DD_{x_j}^c\}$, $j=1,2$. The orientation of $\Omega_j$ is induced by the standard orientation of $\C$, and $\partial\Omega_j^+$ stands for the boundary of $\Omega_j$ with the orientation induced by $\Omega_j$, while $\partial\Omega_j^-$ is the boundary oriented in the opposite way. In the calculations below the orientation of Cartesian products of manifolds is meant to be induced by those of the factors. Denote for short $\partial_j:=\partial_{z_j}$ and $\overline{\partial_j}:=\partial_{\overline{z}_j}$. Since $\Delta_{z_j}=4\partial_{j}\overline{\partial_{j}}$, it holds that
\begin{equation}
\begin{split}
&\int_{\Omega_1} \!\!\!\!\dif^2 z_1 \!\!\int_{\Omega_2} \!\!\!\!\dif^2 z_2 \overline{\Delta f_1(z_1)}\Delta f_2(z_2)\Theta(z_1,z_2)\! =\! -4\!\!\int_{\Omega_1\times \Omega_2} \!\!\!\!\overline{\partial_1\overline{\partial_1}f_1(z_1)} \partial_2\overline{\partial_2} f_2(z_2) \Theta(z_1,z_2) \dif z_1\!\wedge\! \dif \overline{z}_1\! \wedge\! \dif z_2\!\wedge\! \dif \overline{z}_2\\
&\quad=4\int_{\Omega_1\times \Omega_2} \!\!\overline{\partial_1\overline{\partial_1}f_1(z_1)} \partial_2 f_2(z_2) \overline{\partial_2}\Theta(z_1,z_2) \dif z_1\wedge \dif \overline{z}_1 \wedge \dif z_2\wedge \dif \overline{z}_2\\
&\quad + 4\int_{\Omega_1 \times \partial\Omega_2^+} \overline{\partial_1\overline{\partial_1}f_1(z_1)} \partial_2 f_2(z_2) \Theta(z_1,z_2)\dif z_1\wedge \dif \overline{z}_1\wedge \dif z_2.
\end{split}
\label{eq:int_by_parts_1}
\end{equation}
Here we used that $\Omega_1\times \Omega_2$ introduces the orientation $\Omega_1 \times \partial\Omega_2^+$ on this part of the boundary. Integrating by parts in $z_1$, we get that the second line of \eqref{eq:int_by_parts_1} equals to
\begin{equation}
\begin{split}
&-4\int_{\Omega_1\times \Omega_2} \!\!\overline{\partial_1f_1(z_1)} \partial_2 f(z_2) \partial_1\overline{\partial_2}\Theta(z_1,z_2) \dif z_1\wedge \dif \overline{z}_1 \wedge \dif z_2\wedge \dif \overline{z}_2\\
& + 4\int_{\partial \Omega_1^+\times \Omega_2} \overline{\partial_1f_1(z_1)}\partial_2 f_2(z_2) \overline{\partial_2}\Theta(z_1,z_2)\dif\overline{z}_1\wedge \dif z_2 \wedge \dif \overline{z}_2.
\end{split}
\label{eq:int_by_parts_2}
\end{equation}

Let us perform summation over all four regimes $\Omega_j\in\{\DD_{x_j},\DD_{x_j}^c\}$, $j=1,2$, in \eqref{eq:int_by_parts_1}. The integrals in third line over $\Omega_1\times (\partial \DD_{x_2})^+$ and $\Omega_1\times (\partial \DD_{x_2}^c)^+$ cancel out, since these surfaces are oppositely oriented and the integrands are identical. Similarly, after summation over all four integration regimes, the terms in the second line of \eqref{eq:int_by_parts_2} cancel out in the same pairs. Thus, only terms in the first line of \eqref{eq:int_by_parts_2} remain. Observe additionally that
\begin{equation}
\partial_1 \overline{\partial_2}\Theta(z_1,z_2) = 0,\quad\forall z_1\in \DD_{x_1},\, z_2\in \DD_{x_2}^c\quad\text{and}\quad \forall z_1\in \DD_{x_1}^c,\, z_2\in \DD_{x_2}.
\end{equation}
Therefore, \eqref{eq:Theta_int} equals to
\begin{equation}
-4\left(\int_{\DD_{x_1}\times \DD_{x_2}} +\int_{\DD_{x_1}^c\times \DD_{x_2}^c}\right) \overline{\partial_1f_1(z_1)} \partial_2 f_2(z_2) \partial_1\overline{\partial_{2}}\Theta(z_1,z_2) \dif z_1\wedge \dif \overline{z}_1 \wedge \dif z_2\wedge \dif \overline{z}_2.
\label{eq:int_two_regimes}
\end{equation}
The integral over $\DD_{x_1}\times \DD_{x_2}$ gives rise to the first term in the second line of \eqref{eq:V_int}, while in the integral over $\DD_{x_1}^c\times \DD_{x_2}^c$ we perform two more integrations by parts. Since in the regime $z_j\in\DD_{x_j}^c$, $j=1,2$, it holds that $\partial_1\overline{\partial_1}\Theta=\partial_2\overline{\partial_2}\Theta=0$, the integral over the second regime in \eqref{eq:int_two_regimes} (together with the factor $-4$) equals to
\begin{equation}
\begin{split}
&-4\int_{\DD_{x_1}^c\times \partial\DD_{x_2}^-} \overline{\partial_1f_1(z_1)} f_2(z_2) \partial_1\overline{\partial_2}\Theta(z_1,z_2) \dif z_1\wedge \dif \overline{z}_1 \wedge \dif \overline{z}_2\\
&= 4\int_{\partial\DD_{x_1}^-\times \partial \DD_{x_2}^-} \overline{f_1(z_1)} f_2(z_2) \partial_1\overline{\partial_2}\Theta(z_1,z_2) \dif z_1\wedge \dif \overline{z}_2= 4x_1\int_{\partial\DD_{x_1}}\int_{\partial \DD_{x_2}} \frac{\overline{f_1(z_1)}f_2(z_2)}{(x_1-z_1\overline{z}_2)^2}\dif z_1\wedge \dif \overline{z}_2.
\end{split}
\end{equation}
This finishes the proof of Lemma~\ref{lem:UV_int}.
\end{proof}

\section{Proof of Proposition~\ref{prop:underline_comp}}
\label{sec:compcov}

Throughout the proof we use the short-hand notations introduced in Proposition~\ref{prop:underline_comp} and in \eqref{eq:defP}--\eqref{eq:M2def}. We also use $\xi>0$ to denote an $N$-independent exponent which can be taken arbitrarily small. The precise value of $\xi$ may change from line to line. We remind the reader that $x_1,x_2$ are not interchangeable and that $x_1\le x_2$.

Since the proof in the complex and real cases is almost identical, we focus on the complex case and in Section~\ref{sec:Cov_real} explain how to adjust this argument in the real case. The proof of Proposition~\ref{prop:underline_comp} for $\beta=2$ consists of two steps. First, in Section~\ref{sec:Cov_underline} we show that
\begin{equation}
\begin{split}
\E \langle G_1-\E G_1\rangle\langle G_2-\E G_2\rangle  = &\frac{-1}{2x_1x_2N^2} \left(\frac{\ii x_1}{2}\sum_{\sigma\in\{\pm\}}\sigma\partial_{\eta_2}\E\left[\langle G_1AE_\sigma \widetilde{G}_2 E_\sigma\rangle\right] + x_1^2\kappa_4 U_1U_2\right)\\
 + &\mathcal{O}\left(\left(\frac{1}{N\eta_1}+\frac{1}{\sqrt{N}}\right)\frac{N^\xi}{N^2\eta_1\eta_2}\right),
\end{split}
\label{eq:Cov_main}
\end{equation}
where $E_\sigma=E_\sigma^{(k_1)}$ and $A$ is given by
\begin{equation}
A:=\left(1-x_1\langle M_1^2\rangle\right)^{-1}M_1\in\C^{(2k_1)\times (2k_1)}.
\label{eq:def_A}
\end{equation}
Next, in Section~\ref{sec:Cov_determ} we apply the two-resolvent local law from Proposition~\ref{prop:2G_randomD}, and compute the deterministic approximation to the sum over $\sigma$ in the rhs. of \eqref{eq:Cov_main} in terms of the entries of $M_1,M_2$.

\subsection{Proof of \eqref{eq:Cov_main}}\label{sec:Cov_underline}
The main tool used in the proof of \eqref{eq:Cov_main} is the \emph{second order (Gaussian) renormalization} (also known as \emph{underline}), originally introduced in \cite[Eq.(4.2)]{ETH_Wigner} (for earlier works, which did not operate with this formalism but followed a similar approach see e.g. \cite[Eq.(2.5)]{Knowles20}). For any differentiable function~$g(W)$ we denote
\begin{equation}
\underline{Wg(W)}:=Wg(W) - \widetilde{\E}\left[ \widetilde{W}\left(\partial_{\widetilde{W}}g\right)(W)\right],
\label{eq:def_under_1}
\end{equation} 
where $\widetilde{W}$ is an independent copy of $W$, the expectation $\widetilde{\E}$ is taken with respect to $\widetilde{W}$ and $\partial_{\widetilde{W}}$ is the derivative in the direction of $\widetilde{W}$. The name second order renormalization arises from the fact that the underline removes the second order terms in the cumulant expansion for $Wg(W)$. In particular, $\E\underline{Wg(W)}=0$ for $W$ with Gaussian entries.

We start the proof of \eqref{eq:Cov_main} with the basic identity
\begin{equation}
\langle G_1-M_1\rangle = \langle \mathcal{S}^{(k_1)}_x[G_1-M_1](G_1-M_1)A\rangle - \langle \underline{W_1G_1A}\rangle,\quad\text{where}\quad A:=\left(\big(\mathcal{B}_{11}^*\big)^{-1}\big[E_+^{(k_1)}\big]\right)^*M_1.
\label{eq:G_underline}
\end{equation}
For the proof see \cite[Eq.(6.9)]{macroCLT_complex} together with the rescaling \eqref{eq:M_rescaling}. The one-body stability operator $\mathcal{B}_{11}$ is defined in \eqref{eq:def_B_11} and its adjoint is taken with respect to the scalar product $\langle R,S\rangle:=\mathrm{Tr}[R^*S]$ on $\C^{(2k_1)\times (2k_1)}$. An explicit calculation based on~\eqref{eq:def_B_11} shows that $A$ defined in \eqref{eq:G_underline} coincides with \eqref{eq:def_A}, cf. the formula below \cite[Eq.(3.39)]{hyperuniformity} stated for $x_1=1$. To compute the lhs. of \eqref{eq:Cov_main}, we use \eqref{eq:G_underline} and apply \eqref{eq:1G_av} both for $G_1$ and $G_2$ in the term containing three resolvents, arriving at
\begin{equation}
\E \langle G_1-\E G_1\rangle \langle G_2-\E G_2\rangle\!=\! \E \langle G_1-M_1\rangle \langle G_2-\E G_2\rangle\!=\! -\E\left[\langle\underline{W_1G_1A}\rangle \left(\langle G_2\rangle -\E \langle G_2\rangle\right)\right] + \mathcal{O} \left(\frac{N^\xi\|A\|}{N^3\eta_1^2\eta_2}\right)\!. 
\label{eq:underline_comp1}
\end{equation}
In Appendix~\ref{sec:beta} we show that
\begin{equation}
\|A\|\lesssim \left\vert x_1-|z_1|^2\right\vert^{-1},
\label{eq:A_bound}
\end{equation}
so the error term in the rhs. of \eqref{eq:underline_comp1} can be incorporated into the error term in the rhs. of \eqref{eq:Cov_main} (recall that the latter one may deteriorate polynomially in $|x_1-|z_1|^2|$).

Next, we expand the underline on the second factor in the rhs. of \eqref{eq:underline_comp1} by the means of \eqref{eq:def_under_1}. To this end, let $\widetilde{W}_1$ be an independent copy of $W_1$, and let $\widetilde{\E}_1$ be the expectation wrt. $\widetilde{W}_1$. We have
\begin{equation}
\partial_{\widetilde{W}_1}\left(\langle G_2\rangle-\E\langle G_2\rangle\right) = -\langle G_2\partial_{\widetilde{W}_1}[W_2] G_2\rangle = \frac{\ii}{2x_2 N} \partial_{\eta_2} \mathrm{Tr}\left[ \partial_{\widetilde{W}_1}[W_2]G_2\right]= \frac{\ii}{2x_2 N} \partial_{\eta_2} \mathrm{Tr}\left[\widetilde{W}_1 \widetilde{G}_2\right],
\label{eq:underline_comp2}
\end{equation}
where in the last identity we used \eqref{eq:defP} and the fact that $W_1$ is a minor of $W_2$. Observe that for any $B_1,B_2\in \C^{(2k_1)\times (2k_1)}$ independent of $\widetilde{W}_1$ it holds that
\begin{equation}
\widetilde{\E}_1 \left[ \mathrm{Tr}[\widetilde{W}_1B_1]\mathrm{Tr}[\widetilde{W}_2B_2]\right] = \frac{1}{2N}\sum_{\sigma\in\{\pm\}}\sigma \mathrm{Tr}\left[ B_1E_\sigma^{(k_1)}B_2E_\sigma^{(k_1)}\right].
\label{eq:quadr_form}
\end{equation}
Combining \eqref{eq:underline_comp1}--\eqref{eq:quadr_form} with \eqref{eq:def_under_1}, we obtain
\begin{equation}
\begin{split}
\label{eq:usrel}
\E \langle G_1-\E G_1\rangle \langle G_2-\E G_2\rangle = -\frac{\ii}{4x_2 N^2} \sum_{\sigma\in\{\pm\}}\sigma\partial_{\eta_2}\left\langle G_1A E_\sigma^{(k_1)} \widetilde{G}_2 E_\sigma^{(k_1)}\right\rangle-\Upsilon,
\end{split}
\end{equation}
where $\Upsilon$ is the fully underlined term:
\begin{equation}
\Upsilon=\Upsilon\big((x_1,\eta_1,z_1),(x_2,\eta_2,z_2)\big):=\underline{\langle W_1G_1A\rangle \left(\langle G_2\rangle -\E \langle G_2\rangle\right)}.
\end{equation}

To finish the proof of \eqref{eq:Cov_main} it remains to compute $\E \Upsilon$. Note that in the Gaussian case $\E\Upsilon=0$. We will perform a cumulant expansion and see that only cumulants of order at least three contribute to $\E\Upsilon$. Introduce the index set 
\begin{equation}
\begin{split}
\mathcal{I}_1=& \left([(1-x_1)N+1,N] \cup [(2-x_1)N+1,2N]\right)^2,\\
\mathcal{J}_1=& [(1-x_1)N+1,N] \times [(2-x_1)N+1,2N] \cup [(2-x_1)N+1,2N]\times [(1-x_1)N+1,N].
\end{split}
\label{eq:def_IJ}
\end{equation}
We label the entries of $W_1$ and $G_1$ by $\mathcal{I}_1$. Then non-zero entries of $W_1$ are enumerated by $\mathcal{J}_1$. For $l\in \N$ and a collection of index pairs
\begin{equation}
\bm{\alpha}=\left(\alpha_1,\ldots,\alpha_l\right)\in \mathcal{I}_1^l
\label{eq:def_index_set}
\end{equation}
denote the normalized cumulant of the corresponding elements of $W_1$ by
\begin{equation}
\kappa(\bm{\alpha})=\kappa(\alpha_1,\ldots,\alpha_k):=\kappa(\sqrt{N}w_{\alpha_1},\ldots,\sqrt{N}w_{\alpha_k}).
\label{eq:def_cum}
\end{equation}
Since $X$ has i.i.d. entries, $\kappa(\bm{\alpha})$ may not vanish only when $\alpha_1,\ldots,\alpha_k\in\{ab,ba\}$ for some $(a,b)\in\mathcal{J}_1$. We further denote $\partial_{\bm{\alpha}}:=\partial_{\alpha_1}\cdots\partial_{\alpha_k}$, where $\partial_{\alpha}$ denotes the directional derivative in the direction $w_{\alpha}$, for an index pair~$\alpha$.

Performing the cumulant expansion \cite[Lemma~3.2]{Knowles20} (see also \cite[Eq.(7.14)]{hyperuniformity}, \cite[Lemma~3.1]{HeKnowles} and \cite[Section~II]{Khorunzhy96}) in $\E \Upsilon$, we obtain
\begin{equation}
\E\Upsilon = \frac{1}{2x_1N}\sum_{(a,b)\in\mathcal{J}_1}\sum_{\ell=3}^L\sum_{\bm{\alpha}\in\{ab,ba\}^{\ell-1}}\frac{\kappa(ab,\bm{\alpha})}{N^{\ell/2}(\ell-1)!}\E\partial_{\bm{\alpha}}\left[ \left(G_1A\right)_{ba}\left(\langle G_2\rangle - \E\langle G_2\rangle\right)\right] + \mathcal{O}(N^{-D}),
\label{eq:Cov_cum_exp}
\end{equation}
for any fixed $D>0$ and for some $L\in\N$ which depends on $D$, but does not depend on $N$. Here we used that due to the definition of the underline \eqref{eq:def_under_1}, the second order cumulants are absent in the expansion. The analysis of \eqref{eq:Cov_cum_exp} closely follows the one in \cite[Eq.(7.14)-(7.24)]{hyperuniformity}, so we only explain required adjustments. First, \cite[Eq.(7.14)-(7.24)]{hyperuniformity} concerns only the regime $\eta_1,\eta_2\lesssim 1$, however these estimates hold also for $\eta_1\lesssim 1$ and $\eta_2\gtrsim 1$. In the regime $\eta_1\gtrsim 1$ we need to use additionally that
\begin{equation}
\|M_1\|\lesssim \frac{1}{\eta_1}\quad\text{and}\quad \left\vert(G_1-M_1)_{ab}\right\vert\prec\frac{1}{\sqrt{N}\eta_1^2},
\label{eq:large_eta_bounds}
\end{equation}
uniformly in $(a,b)\in\mathcal{I}_x$. The first bound in \eqref{eq:large_eta_bounds} directly follows from the MDE \eqref{eq:MDE} for $M_1$, while the second one is the isotropic local law from \cite[Eq.(5.2)]{loc_circ_law18}. Using these inputs, we get that the terms with $\ell\ge 5$ and $\ell=3$ contribute at most 
\begin{equation}
\mathcal{E}:=\frac{1}{\sqrt{N}}\frac{N^\xi}{N^2\eta_1\eta_2}
\label{eq:cum_exp_bound}
\end{equation}
to the rhs. of \eqref{eq:Cov_cum_exp}. We further have that the $\ell=4$ terms contributing to the leading order in the rhs. of \eqref{eq:Cov_cum_exp} correspond to $\bm{\alpha}$ which is a permutation of $(ab,ba,ba)$, and exactly one derivative $\partial_{ba}$ hits the term $(G_1A)_{ba}$, while the remaining two hit $\langle G_2\rangle-\E\langle G_2\rangle$. The rest of the $\ell=4$ terms can be incorporated into the error term~\eqref{eq:cum_exp_bound}. Thus, it holds that
\begin{equation}
\E\Upsilon =\frac{1}{2x_1N}\sum_{(a,b)\in\mathcal{J}_1} \frac{\kappa_4}{N^2} \E \partial_{ba} [(G_1A)_{ba}]\partial_{ab}\partial_{ba} \left[ \langle G_2\rangle -\E\langle G_2\rangle\right] +\mathcal{O}\left(\mathcal{E}\right).
\label{eq:cum_exp_kappa4}
\end{equation}
Here we took into account that there are 3 permutations of $(ab,ba,ba)$ and 2 ways to distribute derivatives over the two factors (due to the fact that there are two derivatives in $w_{ba}$). This canceled the factor $(\ell-1)!$ for $\ell=4$ in denominator in the rhs. of \eqref{eq:Cov_cum_exp}.

Next, we perform differentiation in \eqref{eq:cum_exp_kappa4} and replace each of the four obtained factors by its deterministic approximation using the isotropic single-resolvent local law \eqref{eq:1G_iso}. This yields an error term which can be again incorporated into the second term in the rhs. of \eqref{eq:cum_exp_kappa4}, see \cite[Eq.(7.18)-(7.19)]{hyperuniformity} with the adjustments explained above. We get
\begin{equation}
\E\Upsilon =\frac{-1}{4x_1x_2N^2}\sum_{(a,b)\in\mathcal{J}_1} \frac{\kappa_4}{N^2} (M_1)_{bb}(M_1A)_{aa}\!\!\left(\!(M_2)_{aa}\big(M_{22}^{E_+}\big)_{bb}\! +\! \big(M_{22}^{E_+}\big)_{aa}(M_2)_{bb}\right) +\mathcal{O}\left(\mathcal{E}\right)\!,
\label{eq:cum_exp_kappa4_determ}
\end{equation}
where $E_+=E_+^{(k_2)}$ and $M_{22}^{E_+}$ is defined in \eqref{eq:def_M_11}. An explicit calculation based on \eqref{eq:MDE} and \eqref{eq:def_B_11} shows that
\begin{equation}
M_{22}^{E_+} = \mathcal{B}_{22}^{-1}[M_2^2]= -\ii\partial_{\eta_2} M_2\quad\text{and}\quad M_1A=\frac{M_1^2}{1-x_1\langle M_1^2\rangle}=\mathcal{B}_{11}^{-1}[M_1^2]=-\ii\partial_{\eta_1} M_1,
\label{eq:MA}
\end{equation}
compare with \cite[Eq.(7.21)]{hyperuniformity} stated for $x_1=x_2=1$. Since all diagonal entries of $M_j$ equal to $m_j$, we conclude that 
\begin{equation}
\begin{split}
\E\Upsilon =&\frac{2}{4x_1x_2N^2}\sum_{(a,b)\in\mathcal{J}_1} \frac{\kappa_4}{N^2} m_1 \partial_{\eta_1}[m_1] m_2 \partial_{\eta_2}[m_2]  +\mathcal{O}\left(\mathcal{E}\right)\\
=&\frac{2|\mathcal{J}_1|\kappa_4}{4x_1x_2N^4} \frac{1}{2}\partial_{\eta_1}[m_1^2]\frac{1}{2}\partial_{\eta_2}[m_2^2]+\mathcal{O}(\mathcal{E}) = -\frac{x^2_1\kappa_4}{2x_1x_2N^2}U_1U_2 +\mathcal{O}(\mathcal{E}),
\end{split}
\label{eq:cum_exp_kappa4_determ}
\end{equation}
where we computed that $|\mathcal{J}_1|=2x_1^2N^2$ and recalled the definition of $U_j$ from \eqref{eq:def_UV}. This finishes the proof of \eqref{eq:Cov_main}.

\subsection{Computation of the deterministic approximation}\label{sec:Cov_determ} In this section we consistently denote $E_\sigma:=E^{(k_1)}_\sigma$, for $\sigma\in\{\pm\}$. We claim that 
\begin{equation}
\left\vert\partial_{\eta_2} \left\langle \left(G_1AE_\sigma \widetilde{G}_2 -M_{1\widetilde{2}}^{E_\sigma}\right)E_\sigma\right\rangle\right\vert \prec\frac{1}{N\eta_1\eta_2^2},
\label{eq:dif_ll}
\end{equation}
where the implicit constant in the rhs. of \eqref{eq:dif_ll} may deteriorate polynomially in $||z_j|-\sqrt{x_j}|$, $j=1,2$. This immediately follows from the local law from Proposition~\ref{prop:2G_randomD} combined with the observation that the lhs. of \eqref{eq:dif_ll} (without the absolute value) is an analytic function of $w_2$ evaluated at $w_2=\ii\eta_2$, so its derivative can be represented as a contour integral via Cauchy's integral formula. For a similar argument see e.g. \cite[Eq.(7.11)]{hyperuniformity}. In particular, because of this step we formulated Proposition~\ref{prop:2G_randomD} not only on the imaginary axis, but for general spectral parameters.  

From \eqref{eq:dif_ll} we have that
\begin{equation}
\sum_{\sigma\in\{\pm\}}\sigma\partial_{\eta_2}\E\left[\langle G_1AE_\sigma \widetilde{G}_2 E_\sigma\rangle\right] = \partial_{\eta_2}\sum_{\sigma\in\{\pm\}} \sigma\left\langle M_{1\widetilde{2}}^{AE_\sigma}E_\sigma\right\rangle + \mathcal{O}_\theta\left(\frac{N^\xi}{N\eta_1\eta_2^2}\right).
\label{eq:Cov_to_determ}
\end{equation}
Here $\mathcal{O}_\theta(\cdot)$ is understood in the same sense as in Proposition~\ref{prop:underline_comp}. Computation of the sum in the rhs. of \eqref{eq:Cov_to_determ} closely follows the argument in \cite[Eq.(7.25)-(7.27)]{hyperuniformity}. First, we define $\mathcal{B}_{\widetilde{2}1}$ by interchanging $M_1$ and $\widetilde{M}_2$ in \eqref{eq:def_B_12}. Similarly, $M_{\widetilde{2}1}^B:=\mathcal{B}_{\widetilde{2}1}^{-1}[\widetilde{M}_2BM_1]$ for an $(2x_1N)\times (2x_1N)$ observable $B$. An elementary calculation based on \eqref{eq:def_B_12}, \eqref{eq:M2def}, and on their analogues which we have just defined, shows that
\begin{equation}
\left\langle M_{1\widetilde{2}}^{AE_\sigma}E_\sigma\right\rangle = \left\langle M_{\widetilde{2}1}^{E_\sigma}AE_\sigma\right\rangle = \left\langle \mathcal{B}_{\widetilde{2}1}^{-1}[\widetilde{M}_2E_\sigma M_1] AE_\sigma\right\rangle,\quad \sigma\in\{\pm\}.
\label{eq:12_to_21}
\end{equation}
The exact choice of observables is not important, this identity holds with $AE_\sigma$ and $E_\sigma$ replaced by general $B_1,B_2\in\C^{(2k_1)\times (2k_1)}$. To compute the rhs. of \eqref{eq:12_to_21}, we denote for short $V_\sigma:=\widetilde{M}_2E_\sigma M_1$ for $\sigma\in\{\pm\}$. By the definition of $\mathcal{B}_{\widetilde{2}1}$ it holds that
\begin{equation}
\begin{pmatrix}
\mathcal{B}_{\widetilde{2}1}[V_+]\\\mathcal{B}_{\widetilde{2}1}[V_-]
\end{pmatrix}
=\mathcal{P}\begin{pmatrix}
V_+\\V_-
\end{pmatrix},\quad\text{where}\quad
\mathcal{P}=
\begin{pmatrix}
1-x_1\langle V_+\rangle&x_1\langle V_+E_-\rangle\\ -x_1\langle V_-\rangle &1+x_1\langle V_-E_-\rangle
\end{pmatrix}.
\label{eq:B_V}
\end{equation}
Thus we have
\begin{equation}
\begin{pmatrix}
\mathcal{B}^{-1}_{\widetilde{2}1}[V_+]\\\mathcal{B}^{-1}_{\widetilde{2}1}[V_-]
\end{pmatrix}
=\mathcal{P}^{-1}\begin{pmatrix}
V_+\\V_-
\end{pmatrix}=\frac{1}{\det \mathcal{P}}
\begin{pmatrix}
\left(1+x_1\langle V_-E_-\rangle\right)V_+ - x_1\langle V_+E_-\rangle V_-\\
x_1\langle V_-\rangle V_+ +\left(1-x_1\langle V_+\rangle\right)V_-
\end{pmatrix}.
\label{eq:P_inverse}
\end{equation}
Combining \eqref{eq:12_to_21} with \eqref{eq:P_inverse} and the second part of \eqref{eq:MA}, we conclude that
\begin{equation}
\sum_{\sigma\in\{\pm\}} \sigma\left\langle M_{1\widetilde{2}}^{AE_\sigma}E_\sigma\right\rangle = \frac{\ii}{x_1}\partial_{\eta_1} \log \det \mathcal{P}.
\end{equation}
An explicit calculation based on \eqref{eq:def_u} and \eqref{eq:M_tilde} shows that $\det\mathcal{P}$ equals to the argument of logarithm in the last line of \eqref{eq:def_UV}. Together with \eqref{eq:Cov_main} and \eqref{eq:Cov_to_determ} this finishes the proof of Proposition~\ref{prop:underline_comp} in the complex case.

\subsection{Real case}\label{sec:Cov_real} There are only three minor differences in the proof of Proposition~\ref{prop:underline_comp} in the real case compared to the complex case. First, in the rhs. of \eqref{eq:G_underline} an additional term $N^{-1}\langle G_1^\mt A\rangle$ arises. This term does not contribute to the leading order in \eqref{eq:underline_comp1} and can be incorporated into the error term similarly to the first term in the rhs. of \eqref{eq:G_underline}. Second, the following new term arises in the rhs. of \eqref{eq:quadr_form}:
\begin{equation}
\frac{1}{2N}\sum_{\sigma\in\{\pm\}}\sigma  \mathrm{Tr}\left[ B_1E_\sigma^{(k_1)}B_2^\mt E_\sigma^{(k_1)}\right].
\label{eq:quadr_form_real}
\end{equation}
As in the complex case, we take $B_1:=G_1A$ and $B_2:=\widetilde{G}_2$. Since $W^{(k_1)}$ is a real symmetric matrix and $P$ defined in \eqref{eq:defP} has real entries, we obtain
\begin{equation}
\widetilde{G}_2^\mt =P^*G_2^\mt P = P^*G^{(k_2),\overline{z}_2}(\ii\eta_2)P,
\end{equation}
i.e. transpose on $\widetilde{G}_2$ results in replacing $z_2$ by $\overline{z}_2$. Thus, \eqref{eq:quadr_form_real} contributes $V_{1\overline{2}}$ to the rhs. of \eqref{eq:Cov_UV}. Finally, in the analysis of the fully underlined term \eqref{eq:Cov_cum_exp}--\eqref{eq:cum_exp_kappa4_determ} the formula for $\kappa_4$ changes according to \eqref{eq:def_kappa_4}, while the entire analysis remains intact. This finishes the proof of Proposition~\ref{prop:underline_comp} in the real case.

\section{Proof of the local law in Proposition~\ref{prop:2G_randomD}}\label{sec:local_law_proof}

Recall the definition of the two-body stability operator $\mathcal{B}_{1\widetilde{2}}$ from \eqref{eq:def_B_12}. One of the main ingredients in the proof of Proposition~\ref{prop:2G_randomD} is the upper bound on $\|\mathcal{B}_{1\widetilde{2}}^{-1}\|$, since this quantity controls the size of $M_{1\widetilde{2}}^B$ via \eqref{eq:M2def}. Thus, we start with the discussion of spectral properties of $\mathcal{B}_{1\widetilde{2}}$. 

Throughout this section we use the following short-hand notations. Index $j$ is always meant to be from $\{1,2\}$. Denote $M_j:=M^{(k_j),z_j}_{x_j}(w_j)$ for $x_j:=k_j/N\in (0,1]$ with $x_1\le x_2$, $z_j\in \C$ and $w_j\in\C\setminus\R$. The same convention is used for $m_j, u_j$ and $\rho_j$ defined in \eqref{eq:def_u} and \eqref{eq:def_rho}. Denote additionally $\eta_j:=|\Im w_j|$. Whenever we refer to some properties of \eqref{eq:MDE} known from the literature, we mean that they are known for $x=1$ and for general $x\in (0,1]$ one needs to apply the rescaling \eqref{eq:M_rescaling}.  

First, we have that $\mathcal{B}_{1\widetilde{2}}$ has eigenvalue 1 of multiplicity $(2k_1)^2-2$, and the remaining two eigenvalues are given by
\begin{equation}
\beta_\pm := 1-x_1\Re [z_1\overline{z}_2] u_1u_2 \pm x_1\sqrt{m_1^2m_2^2 -(\Im [z_1\overline{z}_2]u_1u_2)^2}.
\label{eq:beta_expl}
\end{equation}
The proof of \eqref{eq:beta_expl} immediately follows from the fact that the eigenvalues of $\mathcal{B}_{1\widetilde{2}}$ different from 1 are the eigenvalues of the $2\times 2$ matrix $\mathcal{P}$ defined in \eqref{eq:B_V}. For a similar argument see \cite[Eq.(6.3)-(6.4)]{macroCLT_complex}. It further holds that
\begin{equation}
\|\mathcal{B}_{1\widetilde{2}}^{-1}[R]\|\lesssim \frac{\|R\|}{\beta_*},\quad\forall\, R\in\C^{(2k_1)\times(2k_1)},\quad \text{where}\quad \beta_*:=\min\{|\beta_+|,|\beta_-|\}.
\label{eq:B_inv_bound}
\end{equation}
The implicit constant in \eqref{eq:B_inv_bound} depends only on $c_0>0$ in the condition $x_1,x_2\ge c_0$. Since $\mathcal{B}_{1\widetilde{2}}$ is not Hermitian, \eqref{eq:B_inv_bound} is not a direct consequence of \eqref{eq:beta_expl}. However, \eqref{eq:B_inv_bound} follows by a explicit inversion of $\mathcal{B}_{1\widetilde{2}}$, see e.g. \cite[Lemma~S2.8(1)]{hyperuniformity}.

To get an upper bound on the lhs. of \eqref{eq:B_inv_bound} it suffices to get a lower bound on $\beta_*$. Since $|m_j|, |u_j|\lesssim 1$ by \cite[Eq.(3.5)]{macroCLT_complex}, we have $\beta_*\gtrsim |\beta_+\beta_-|$. We further focus on the analysis of the latter quantity, since it is more accessible than $\beta_*$. Let us immediately remove from consideration the regime when at least one of $\eta_1,\eta_2$ is large. Indeed, by the definition of $u_j$ \eqref{eq:def_u}, $|u_j|\lesssim \eta_j^{-1}$. Thus, there exists a (large) constant $C>0$ such that if $\eta_j>C|z_j|$ for at least one $j$, we have $\beta_*\sim 1$ by \eqref{eq:beta_expl}. 

Consider first the case when $w_1,w_2$ are purely imaginary.

\begin{lemma}[$\beta_+\beta_-$ on the imaginary axis]\label{lem:beta} Fix a (small) $c_0>0$ and a (large) $L>0$. For $w_j=\pm\ii\eta_j$ it holds that
\begin{equation}
|\beta_+\beta_-|\gtrsim \rho_2^2|x_2-x_1|+|z_1-z_2|^2 +\rho_1\eta_1 +\rho_2\eta_2,
\label{eq:min_beta}
\end{equation}
uniformly in $x_j\in [c_0,1]$, $x_1\le x_2$, $|z_j|\le L$ and $0<\eta_j\le L$. Additionally, in the regime $\sqrt{x_j}\le |z_1|\le L$, $j=1,2$, we have
\begin{equation}
|\beta_+\beta_-|\gtrsim \left(|z_1|^2-x_1\right)\left(|z_2|^2-x_2\right).
\label{eq:min_beta_outside}
\end{equation}
\end{lemma}

The proof of Lemma~\ref{lem:beta} is presented in Appendix~\ref{sec:beta}. In the case $x_1=x_2=1$, \eqref{eq:min_beta} was proven in \cite[Lemma~6.1]{macroCLT_complex} in a slightly weaker form in terms of $\rho_j$. Compared to this result, \eqref{eq:min_beta} features a new effect, which is the dependence on $|x_2-x_1|$. In particular, for $|x_2-x_1|\sim 1$ we have $\|\mathcal{B}_{1\widetilde{2}}^{-1}[R]\|\lesssim \|R\|$ even for small $\eta_1,\eta_2$.

Recall that Proposition~\ref{prop:2G_randomD} is stated for $w_1,w_2$ not only on the imaginary axis, but in a narrow cone around it. Thus, we further consider the perturbative regime when $w_1,w_2$ are close to the imaginary axis.

\begin{lemma}[$\beta_+\beta_-$ in the perturbative regime]\label{lem:beta_perturb} Fix a (small) $c_0>0$ and a (large) $L>0$. For $w_j=E_j\pm\ii\eta_j$ it holds that 
\begin{equation}
\beta_+\beta_-(w_1,w_2) = \beta_+\beta_-(\ii\Im w_1,\ii\Im w_2) +\mathcal{O}\left(\frac{|E_1|}{\eta_1}+\frac{|E_2|}{\eta_2}\right),
\label{eq:beta_perturb}
\end{equation}
uniformly in $x_1,x_2\in [c_0,1]$, $x_1\le x_2$, $|z_j|\le L$ and $w_j=E_j\pm\ii\eta_j$ with $-L\le E_j\le L$, $0<\eta_j\le L$, $j=1,2$. Here by $\beta_+\beta_-(w_1,w_2)$ we abbreviated the dependence of $\beta_+\beta_-$ on $w_1,w_2$, while the parameters $z_1,z_2$ are the same for both $\beta_+\beta_-$ terms in \eqref{eq:beta_perturb}.
\end{lemma}

The proof of Lemma~\ref{lem:beta_perturb} is presented in Appendix~\ref{sec:beta}. Now we are ready to proceed to the proof of the local law~\eqref{eq:2G_av}.

\begin{proof}[Proof of Proposition~\ref{prop:2G_randomD}] To prove the local law we rely on the zigzag strategy (see the references below Proposition~\ref{prop:2G_randomD}). Our proof is analogous to the proof of averaged two-resolvent averaged local laws \cite[Eq.(3.26)]{nonHermdecay} and \cite[Eq.(3.16)]{univ_extr}, with only minor modifications that we now explain. In fact, the current proof is simpler compared to the proof of \cite[Eq.(3.26)]{nonHermdecay} and \cite[Eq.(3.16)]{univ_extr} for two reasons. First, we now focus on the bulk of the spectrum. Though Proposition~\ref{prop:2G_randomD} also covers the complementary regime, it is not done optimally, as it is not needed for the proof of our main results. Second, we do not keep track of the additional smallness coming from the parameter $\gamma_s$, using the notation therein, which gives a smaller error when $z_1, z_2$ are away from each other. The parameter $\gamma_s$, which can be thought as an analog of the rhs. of \eqref{eq:min_beta} for $x=y=1$, determines an improved local law bound in terms of the decay $|z_1-z_2|^{-2}$ (cf. \cite[Eq.(3.26)]{nonHermdecay} with \eqref{eq:2G_av}), i.e. further away $z_1, z_2$ give an improvement in the local law bound compared to the case $z_1=z_2$. 

From now on we assume that $\theta\sim 1$ in Proposition~\ref{prop:2G_randomD}, i.e. that $z_i$ is either in the bulk regime or well outside of the non-Hermitian spectrum, for $i=1,2$. In the case when $\theta$ vanishes with $N$, one can see that the estimates below deteriorate polynomially in $\theta$, which is affordable due to the $\theta$-dependence of the rhs. of \eqref{eq:2G_av}. Recall the notations from Section~\ref{sec:strategy_Cov}. When both parameters $z_1,z_2$ satisfy $|z_i|\ge \sqrt{x_i}+\theta$, then $|\beta_+\beta_-|\sim 1$ by \eqref{eq:min_beta_outside} for $w_1,w_2$ on the imaginary axis. Meanwhile, Lemma~\ref{lem:beta_perturb} extends this estimate to a narrow cone $|\Re w_j|\le c\eta_j$, $j=1,2$, for some small $c>0$. Therefore, $\|\mathcal{B}_{1\widetilde{2}}^{-1}[R]\|\lesssim \|R\|$ for any $R\in\C^{(2k_1)\times (2k_1)}$, and the result in \eqref{eq:2G_av} immediately follows by a global law as in  \cite[Appendix~B]{Multi_res_llaws}. In the case when $|z_i|\le \sqrt{x_i}-\theta$ and $|z_j|\ge \sqrt{x_j}+\theta$ for $\{i,j\}=\{1,2\}$, the argument is similar and relies on \eqref{eq:min_beta} instead of \eqref{eq:min_beta_outside}.

We now consider the bulk regime $|z_j|\le \sqrt{x_j}-\theta$. Fix a small $\delta>0$. If at least one of the bounds
\begin{equation}
|x_1-x_2|<N^{-\delta},\quad |z_1-z_2|<N^{-\delta},\eta_j<N^{-\delta}
\label{eq:global_regime}
\end{equation}
does not hold, then by \eqref{eq:min_beta}, \eqref{eq:beta_perturb} we have $|\beta_+\beta_-|\gtrsim N^{-\delta}$, so the global law-type argument gives (see also \cite[Appendix~B]{Multi_res_llaws} for the gain when either $\eta_i\gg 1$)
\begin{equation}
\left\vert\left\langle \left(G_1B_1\widetilde{G}_2 - M_{1\widetilde{2}}^{B_1}\right)B_2\right\rangle\right\vert \le \frac{N^{\xi+C\delta}}{N\eta_1\eta_2}.
\end{equation}
for some fixed $C>0$ and any fixed $\xi>0$. We may further assume that \eqref{eq:global_regime} holds. We are thus left with showing that this bound can be propagated down to $\min_i|\eta_i|\ge N^{-1+\epsilon}$. For this purpose, we embed the matrix $X$ in an Orstein-Uhlenbeck flow and prove \eqref{eq:2G_av} for matrices with an (almost) order one Gaussian component. The fact that this Gaussian component can be removed at the price of a negligible error follows from a standard Green's function comparison (GFT) argument;  see e.g. \cite[Section~5]{nonHermdecay}. To be precise, the GFT argument \cite[Section~5]{nonHermdecay} is not used only once to remove the order one Gaussian component in one step, but rather repeated many times starting for $\eta=N^{-\delta}$ and decreasing the value of $\eta$ by a factor $N^{-\delta}$ at each step (at each step the Gaussian component to be removed decreases by a factor $N^{-\delta}$); this is what is called \emph{zigzag strategy} (see \cite[Fig.~1-2]{nonHermdecay} for an illustration of this approach).

Fix $T=C N^{-\delta}$, for some $C>0$ to be chosen shortly. We now consider the Orstein-Uhlenbeck flow
\begin{equation}
\label{eq:OU}
\dif X_t =-\frac{1}{2} X_t +\frac{1}{\sqrt{N}}B_t, \qquad\quad X_0=X,
\end{equation}
and the associated \emph{characteristic flow} (see e.g. \cite[Eq.(5.3)-(5.4)]{Cipolloni_meso})
\begin{equation}
\frac{\dif}{\dif t} w_{i,t} = -x_i m_{i,t}-\frac{w_{i,t}}{2},\quad\text{with}\quad m_{i,t}:=m_{x_i}^{z_{i,t}}(w_{i,t}), \qquad\quad \frac{\dif}{\dif t} z_{i,t}=-\frac{z_{i,t}}{2},
\label{eq:char_flow}
\end{equation}
for $i=1,2$. We also denote $\eta_{i,t}:=|\Im w_{i,t}|$. The flow \eqref{eq:OU} is considered for times $t\in [0,T]$, so that for each $w_1,w_2\in\C\setminus\R$ with $\eta_1,\eta_2 \ge N^{-1+\epsilon}$ it possible to choose initial conditions $w_{i,0}$ with $\eta_{i,0}\ge N^{-\delta}$ such that the characteristics $w_{i,t}$ with initial data $w_{i,0}$ satisfy $w_{i,T}=w_i$. It is thus enough to prove \eqref{eq:2G_av} for the resolvents of the Hermitization of the solution of \eqref{eq:2G_av}, with hermitization parameters $z_{i,t}$ and spectral parameters $w_{i,t}$, for any $t\in [0,T]$. 

The following lemma shows that if a target $w_{i,T}$ of the characteristic flow \eqref{eq:char_flow} is chosen in a narrow cone around the imaginary axis, then the entire trajectory still lies in a cone which may be slightly broader.  Since this statement holds both for $i=1$ and $i=2$, we omit the index $i$ for brevity.

\begin{lemma}[Trajectories of the characteristic flow]\label{lem:trajectories} Let $T>0$ be chosen as above. Consider $|z_T|\le \sqrt{x}-\theta$ and $w_T\in\C\setminus\R$. Assume that $|\Re w_T|\le c\eta_T$ for a sufficiently small $c>0$.  Then for any $t\in [0,T]$ it holds that $|\Re w_t|\le c'\eta_t$ for some $c'>0$. The constant $c'>0$ can be chosen arbitrarily small provided that $c>0$ is also arbitrary small.
\end{lemma}

This result is needed for the application of Lemma~\ref{lem:beta_perturb} effectively along the trajectory. For the proof of Lemma~\ref{lem:trajectories} see Appendix~\ref{sec:beta}.

In the following, for $i=1,2$, we use the short-hand notations
\begin{equation}
G_{i,t}:=G_{x_i}^{(k_i),z_i}(w_i),\,\, i=1,2,\qquad\quad \widetilde{G}_{2,t}:=P^*G_{2,t}P,
\end{equation}
and $B_t^{(k_i)}$ to denote the $k_i$-minor of the matrix-valued Brownian motion $B_t$. Additionally, by $M_{1\widetilde{2},t}^B$ we denote the deterministic approximation of $G_{1,t}B\widetilde{G}_{2,t}$, which is defined as in \eqref{eq:M2def} with $M_1$, $\widetilde{M}_2$ in \eqref{eq:def_B_12} and \eqref{eq:M2def} being replaced by
\[
M_{1,t}:=M^{z_{1,t}}_{x_i}(w_{1,t}),\,\,i=1,2, \qquad\text{and}\qquad \widetilde{M}_{2,t}:=P^*M_{2,t}P,
\]
respectively. Note that analogously to \cite[Eq.(A.19),(A.26)]{Cipolloni_meso}, for $i=1,2$ and $t\in[0,T]$, we obtain
\begin{equation}
\label{eq:M2derivative}
\partial_t M_{i,t}=\frac{M_{i,t}}{2}, \qquad\quad \partial_t \left\langle M_{1\widetilde{2},t}^{B_1}B_2\right\rangle =\left\langle M_{1\widetilde{2},t}^{B_1}B_2\right\rangle+\left\langle \mathcal{S}^{(k_1)}_{x_1}\left[M_{1\widetilde{2},t}^{B_1}\right] M_{\widetilde{2}1,t}^{B_2}\right\rangle.
\end{equation}
Here $M_{\widetilde{2}1,t}^{B_1}$ is the time-dependent version of $M_{\widetilde{2}1}^{B_1}$ defined below \eqref{eq:Cov_to_determ}.

Next, using It\^{o}'s formula and \eqref{eq:M2derivative}, for any deterministic matrices $B_1,B_2\in\C^{(2k_1)\times (2k_1)}$ we obtain 
\begin{equation}
\begin{split}
\label{eq:new2Geq}
&\dif \left\langle (G_{1,t}B_1P^*G_{2,t}P-M_{1\widetilde{2},t}^{B_1})B_2\right\rangle = \left\langle (G_{1,t}B_1P^*G_{2,t}P-M_{1\widetilde{2},t}^{B_1})B_2\right\rangle \dif t \\
&\qquad\quad+ \left\langle G_{1,t}\mathcal{S}^{(k_1)}_{x_1}\left[ G_{1,t}B_1P^*G_{2,t}P\right] P^*G_{2,t}PB_2\right\rangle\dif t-\left\langle \mathcal{S}^{(k_1)}_{x_1}\left[M_{1\widetilde{2},t}^{B_1}\right] M_{\widetilde{2}1,t}^{B_2}\right\rangle\\
&\qquad\quad+ x_1\left\langle G_{1,t}-M_{1,t}\right\rangle \left\langle G_{1,t}^2 B_1P^*G_{2,t}PB_2\right\rangle \dif t + x_2\left\langle G_{2,t}-M_{2,t}\right\rangle \left\langle G_{1,t}B_1P^*G_{2,t}^2PB_2\right\rangle\\
&\qquad\quad- \frac{1}{\sqrt{N}}\left(\left\langle G_{1,t}\dif B^{(k_1)}_tG_{1,t}B_1P^*G_{2,t}PB_2\right\rangle + \left\langle G_{1,t}B_1P^*G_{2,t}\dif B_t^{(k_2)}G_{2,t}PB_2\right\rangle\right).
\end{split}
\end{equation}
In \eqref{eq:new2Geq} we wrote out $\widetilde{G}_2$ in terms of $G_2$ and $P$ for transparency. This equation is analogous to \cite[Eq.(4.20)]{nonHermdecay}, with the only exceptions being the appearance of $x_1, x_2$, and $P$ in \eqref{eq:new2Geq}, which do not cause any changes. In fact, the pre-factors $x_1,x_2$ in the third line of \eqref{eq:new2Geq} are of order one, and the $P$ appearing in all the traces can just be removed using a Schwarz inequalities. We just explain this for a specific term, i.e. for the second term in the third line we use 
\begin{equation}
\begin{split}
\left\vert\langle G_{1,t}B_1P^*G_{2,t}^2PB_2\rangle \right\vert &\lesssim \langle P^*G_{2,t}^* G_{2,t}P\rangle^{1/2} \langle B_2G_{1,t}B_1P^*G_{2,t}G_{2,t}^*PB_1^*G_{1,t}^*B_2^*\rangle^{1/2}\\
&\lesssim \frac{\lVert PB_1^*G_{1,t}^*B_2^*B_2G_{1,t}B_1P^*\rVert^{1/2}}{\eta_{2,t}}\left\vert\langle \Im G_{2,t}\rangle\right\vert\lesssim \frac{1}{\eta_{1,t}\eta_{2,t}},
\end{split}
\label{eq:propagator_entries}
\end{equation}
with very high probability, where in the last inequality we used the single resolvent local law from Theorem~\ref{theo:1Gllaw}. This is exactly the same bound one would have obtained if $P$ was not present (see e.g.\footnote{We stress that in the current case we do not need to extract the gain from the parameter $\gamma_s$, using the notation in these references.} \cite[Eq.(4.25)]{nonHermdecay} and \cite[Eq.(5.40)]{univ_extr}). All the other terms with $P$ can be treated analogously. 

For the analysis of \eqref{eq:new2Geq} we need the following input, which is the analogue of \cite[Eq.(4.27)]{nonHermdecay}. The proof of this lemma is postponed to Appendix~\ref{sec:beta}.
 
\begin{lemma}\label{lem:M2_bound} Consider $|z_j|\le\sqrt{x_j}-\theta$ and $w_j\in \C\setminus\R$ for $j=1,2$. For any $B_1,B_2\in\C^{(2k_1)\times (2k_1)}$ it holds that
\begin{equation}
x_1\left\vert\left\langle M_{1\widetilde{2}}^{B_1}B_2\right\rangle\right\vert \le \|B_1\|\|B_2\|\sqrt{\frac{x_1\pi\rho_1}{\eta_1}\frac{x_2\pi\rho_2}{\eta_2}}.
\label{eq:M2_bound1}
\end{equation}
Assume further that $|\Re w_j|\le c\eta_j$ for sufficiently small $c>0$. Then we have
\begin{equation}
x_1\left\vert \left\langle M_{1\widetilde{2}}^{E_\sigma}B\right\rangle\right\vert + x_1\left\vert \left\langle M_{1\widetilde{2}}^{B}E_\sigma\right\rangle\right\vert \lesssim \left(|x_1-x_2|+|z_1-z_2|+|w_1|+|w_2|\right)\|B\|\sqrt{\frac{x_1\pi\rho_1}{\eta_1}\frac{x_2\pi\rho_2}{\eta_2}},
\label{eq:M2_bound2}
\end{equation}
for $\sigma:=\mathrm{sgn}(\Im w_1\Im w_2)$ and $E_\sigma:=E_\sigma^{(k_1)}$. This bound holds uniformly in $x_j\in [c_0,1]$ with $x_1\le x_2$, $|z_j|< \sqrt{x_j}$, for any fixed $c_0>0$.
\end{lemma}

Combining \eqref{eq:global_regime} with the bound $|\Re w_{j,t}|\le c'\eta_{j,t}$ from Lemma~\ref{lem:trajectories}, we get that  the first factor in the rhs. of \eqref{eq:M2_bound2} is at most of order $N^{-\delta}$ along the entire trajectory. Observe additionally that 
\begin{equation}
\int_s^t \sqrt{\frac{x_1\pi\rho_{1,r}}{\eta_{1,r}}\frac{x_2\pi\rho_{2,r}}{\eta_{2,r}}}\dif r \le  \frac{1}{2}\int_s^t \left(\frac{x_1\pi\rho_{1,r}}{\eta_{1,r}}+ \frac{x_2\pi\rho_{2,r}}{\eta_{2,r}}\right)\dif r \nc \le\frac{1}{2}\log \frac{\eta_{1,s}\eta_{2,s}}{\eta_{1,t}\eta_{2,t}},\quad \forall\, 0\le s\le t\le T.
\end{equation}
Here in the second bound we used that $x_j\pi\rho_{j,r}=-\partial_r \eta_{j,r}-\eta_{j,r}/2\le-\partial_r \eta_{j,r}$ by \eqref{eq:char_flow}. \nc Given these inputs,the analysis of \eqref{eq:new2Geq} leading to
\begin{equation}
\big\vert\big\langle \big(G_{1,t}B_1\widetilde{G}_{2,t} - M_{1\widetilde{2},t}^{B_1}\big)B_2\big\rangle\big\vert \le \frac{N^{\xi+C\delta}}{N\eta_{1,t}\eta_{2,t}},
\label{eq:ll_proof_fin}
\end{equation}
is analogous to the proof in \cite[Proposition~4.3]{nonHermdecay} (in fact simpler since here we do not carry the dependence of the error on $\gamma_s$, using the notation therein). We omit the details as they are completely analogous. Since $\delta,\xi>0$ can be chosen arbitrarily small in \eqref{eq:ll_proof_fin}, this concludes \eqref{eq:2G_av} for matrices with a Gaussian component of size $T$, and thus the proof.
\end{proof}

\appendix

\section{Proofs of additional technical results}

\subsection{Properties of the bilinear form $C^{(G)}$}\label{sec:C_properties} 

In this section we show that the integrals in \eqref{eq:C_Gauss} are well-defined and prove Corollary~\ref{cor:general_testfn}.

Consider first the case when $x_1\neq x_2$ and assume without loss of generality that $x_1<x_2$. We have that $K>c$ for any $z_j\in\DD_{x_j}$, $j=1,2$ and some $c>0$ depending on $x_1,x_2$. Indeed, both terms in the definition of $K$ are non-negative, and the second one vanishes only when $|z_2|=\sqrt{x_2}$, in which case $|z_1-z_2|\ge |\sqrt{x_1}-\sqrt{x_2}|>0$. Together with the compactness argument this implies the desired bound on $K$. Therefore, the integral in the first line of \eqref{eq:C_Gauss} is not singular and thus is well-defined under the integrability assumption on the derivatives of $f_1,f_2$. The integral in the second line of \eqref{eq:C_Gauss} is also well-defined, since
\begin{equation}
|z_1\overline{z}_2|=\sqrt{x_1x_2}>\min\{x_1,x_2\},\quad \forall \, z_j\in \partial\DD_{x_j},\, j=1,2.
\end{equation}

In the case $x_1=x_2\in(0,1]$ the integral in the second line of \eqref{eq:C_Gauss} is divergent, and \eqref{eq:C_Gauss} is understood in the limiting sense:
\begin{equation}
C^{(G)}\big((x,f_1),(x,f_2)\big):=\!\!\lim_{n\to\infty}C^{(G)}\big((x_1^{(n)}\!\!,f_1),(x_2^{(n)}\!\!,f_2)\big) = \frac{1}{4\pi}\langle \nabla f_1,\nabla f_2\rangle_{L^2(\DD_x)} + \frac{1}{2}\langle f_1,f_2\rangle_{{\dot{H}}^{1/2}(\partial\DD_x)},
\label{eq:C_limit}
\end{equation}
where the limit is taken with respect to any sequence $\{(x_1^{(n)},x_2^{(n)})\}\subset (0,1]$ converging to $(x,x)$, such that $x_1^{(n)}\neq x_2^{(n)}$ for any $n\in\N$. In \eqref{eq:C_limit} we denoted
\begin{equation}
\begin{split}
\langle \nabla f_1,\nabla f_2\rangle_{L^2(\DD_x)}:=&\int_{\DD_x} \langle \nabla f_1(z),\nabla f_2(z)\rangle \dif^2 z,\\
\langle f_1,f_2\rangle_{{\dot{H}}^{1/2}(\partial\DD_x)}:=&\sum_{k\in\mathbf{Z}} |k|\overline{\widehat{f_1\big|_{\partial\DD_x}}(k)}\widehat{f_2\big|_{\partial\DD_x}}(k),\quad \widehat{f_j\big|_{\partial\DD_x}}:=\frac{1}{2\pi} \int_0^{2\pi} f_j(\sqrt{x}\ee^{\ii\theta})\ee^{-\ii\theta k}\dif\theta,\,\, k\in\mathbf{Z}.
\end{split}
\end{equation}
We prove \eqref{eq:C_limit} later in this section. The second identity in \eqref{eq:C_limit} shows that the limit in \eqref{eq:C_limit} exists and does not depend on the choice of $\{(x_1^{(n)},x_2^{(n)})\}$, so the lhs. of \eqref{eq:C_limit} is well-defined. We also note that the rhs. of \eqref{eq:C_limit} is exactly the Gaussian contribution appearing in \cite[Eq.(2.6)]{macroCLT_complex} in the covariance of linear statistics of the same i.i.d. matrix.

To prove \eqref{eq:C_limit} and Corollary~\ref{cor:general_testfn}, it is convenient to write $C^{(G)}$ in the form
\begin{equation}
C^{(G)}\big((x_1,f_1),(x_2,f_2)\big) = \frac{1}{8\pi^2} \int_\C\dif^2 z_1\int_\C \dif^2 z_2 \overline{\Delta f_1(z_1)}\Delta f_2(z_2) \Theta \big((x_1,z_1),(x_2,z_2)\big),
\label{eq:C_Gauss_app}
\end{equation}
as it follows from Lemma~\ref{lem:UV_int}. Here we recalled the definition of $\Theta$ from \eqref{eq:def_Theta} and used \eqref{eq:V_to_Theta} to integrate $V_{12}$ over $\eta_1,\eta_2$. The rhs. of \eqref{eq:C_Gauss_app} is well-defined for any $x_1,x_2\in (0,1]$ and $f_1,f_2\in H^2_0(\Omega)$ with $\Omega$ defined in Theorem~\ref{theo:main}. In the arguments below we prefer to use \eqref{eq:C_Gauss_app} as the definition of $C^{(G)}$ instead of \eqref{eq:C_Gauss}, since this simplifies the analysis.

We derive both \eqref{eq:C_limit} and Corollary~\ref{cor:general_testfn} from the following lemma. 

\begin{lemma}[Regularity of $C^{(G)}$ in $x_1,x_2$]\label{lem:C_reg} Let $\Omega$ be defined as in Theorem~\ref{theo:main}. Fix (small) $c_0,\varepsilon>0$. For any $f_1,f_2\in H_0^{2}(\Omega)$ and $x_j,x_j'\in [c_0,1]$, $j=1,2$, it holds that
\begin{equation}
\left\vert C^{(G)}\big((x_1,f_1),(x_2,f_2)\big) - C^{(G)}\big((x_1',f_1),(x_2',f_2)\big)\right\vert \lesssim \left(|x_1-x_1'|+|x_2-x_2'|\right)^{1/2-\varepsilon} \prod_{j=1}^2\|\Delta f_j\|_{L^2(\Omega)},
\label{eq:C_reg}
\end{equation}
where the implicit constant depends only on $c_0$ and $\varepsilon$.
\end{lemma}

\begin{proof}[Proof of Lemma~\ref{lem:C_reg}] By \eqref{eq:C_Gauss_app} it holds that the lhs. of \eqref{eq:C_reg} has an upper bound of order
\begin{equation}
\int_\C\dif^2 z_1\int_\C \dif^2 z_2 |\Delta f_1(z_1)||\Delta f_2(z_2)| \left\vert \Theta\big((x_1,z_1),(x_2,z_2)\big)-\Theta\big((x_1,z_1),(x_2,z_2)\big)\right\vert.
\label{eq:C_difference}
\end{equation}
Denote the first $\Theta(...)$ in \eqref{eq:C_difference} simply by $\Theta(z_1,z_2)$ and the second by $\Theta'(z_1,z_2)$. We further denote
\begin{equation}
\mathcal{A}_j:=\big\lbrace z_j\in \C: \min\{x_j,x_j'\}\le |z_j|\le \max\{x_j,x_j'\}\big\rbrace,\quad j=1,2.
\end{equation}
The integral in \eqref{eq:C_difference} over $z_1\in \mathcal{A}_1$ and $z_2\in\C$ has an upper bound
\begin{equation}
\begin{split}
&\int_{\mathcal{A}_1}\dif^2 z_1|\Delta f_1(z_1)| \|\Delta f_2\|_{L^2(\Omega)} \|(\Theta(z_1,\cdot)-\Theta'(z_1,\cdot)\|_{L^2(\Omega)}\\
&\quad \le |\mathcal{A}_1|^{1/2} \sup_{z_1\in \Omega}\|(\Theta(z_1,\cdot)-\Theta'(z_1,\cdot)\|_{L^2(\Omega)}\prod_{j=1}^2 \|\Delta f_j\|_{L^2(\Omega)}.
\end{split}
\label{eq:Theta_int_bound}
\end{equation}
The $L^2$ norm in the first line of \eqref{eq:Theta_int_bound} is taken with respect to $z_2$. 
Since $\Theta$ may have only $\log$ singularities, $\Theta(z_1,\cdot)\in L^2(\Omega)$. We also have from \eqref{eq:def_Theta} that the supremum in \eqref{eq:Theta_int_bound} is bounded from above by some universal constant. Together with the bound $|\mathcal{A}_1|\lesssim |x_1-x_1'|$, \eqref{eq:Theta_int_bound} implies that the contribution from the integration regime $z_1\in\mathcal{A}_1$ in \eqref{eq:C_difference} is bounded by the rhs. of \eqref{eq:C_reg}. For $z_2\in\mathcal{A}_2$ the argument is analogous.

We are left with estimating the integral in \eqref{eq:C_difference} over four regimes $\Omega_1\times \Omega_2$, where 
\begin{equation}
\Omega_j\in \left\lbrace \DD_{\min\{x_j,x_j'\}}, \DD_{\max\{x_j,x_j'\}}^c\right\rbrace,\quad j=1,2.
\end{equation}
Here the superscript $c$ denotes the complement of a set. Since for all these regimes the proofs are similar, we consider only the exemplary case $\Omega_j:=\DD_{\min\{x_j,x_j'\}}$, $j=1,2$. In this regime both $\Theta$ and $\Theta'$ are given by the first line of \eqref{eq:def_Theta}. The contribution to \eqref{eq:C_difference} from the second term in this line is immediately estimated from above by the rhs. of \eqref{eq:C_reg}. Thus, we further focus on the first term in the first line of \eqref{eq:def_Theta}. Following the same convention as discussed below \eqref{eq:C_difference}, we introduce the notations $K(z_1,z_2)$ and $K'(z_1,z_2)$. For any fixed $\varepsilon\in (0,1/2)$ it holds that
\begin{equation}
\left\vert\log K - \log K'\right\vert = \log \left(1 + \frac{|K-K'|}{\min\{K,K'\}}\right) \lesssim \left(\frac{|K-K'|}{\min\{K,K'\}}\right)^{1/2-\varepsilon}\!\!\!\!\lesssim \left(\frac{|x_1-x_1'|+|x_2-x_2'|}{|z_1-z_2|^2}\right)^{1/2-\varepsilon}\!\!,
\label{eq:K_difference}
\end{equation}
uniformly in $z_j\in\Omega_j$. In the last step in \eqref{eq:K_difference} we used the explicit form of $K$ from \eqref{eq:def_K}. Since the singularity $|z|^{-2+2\varepsilon}$ is integrable near origin in $\C$, \eqref{eq:K_difference} implies
\begin{equation}
\sup_{z_1\in\Omega}\left\lVert\log K(z_1,\cdot)-\log K'(z_1,\cdot)\right\rVert_{L^2(\Omega_2)}\lesssim \left(|x_1-x_1'|+|x_2-x_2'|\right)^{1/2-\varepsilon}.
\label{eq:sup_L2_bound}
\end{equation}
Finally, we estimate the integral in \eqref{eq:C_difference} in the regime $\Omega_1\times\Omega_2$ similarly to \eqref{eq:Theta_int_bound}, and gain smallness from \eqref{eq:sup_L2_bound} instead of the volume factor. This completes the proof of Lemma~\ref{lem:C_reg}.
\end{proof}

Now we are ready to prove \eqref{eq:C_limit}.

\begin{proof}[Proof of \eqref{eq:C_limit}] Lemma~\ref{lem:C_reg} implies that the limit in \eqref{eq:C_limit} equals to the integral in the rhs. of \eqref{eq:C_Gauss_app} with $x_1=x_2=x$. It is known from \cite[Lemma~4.8]{macroCLT_complex} that this integral equals to the rhs. of \eqref{eq:C_limit} in the special case $x=1$. In the general case $x\in (0,1]$ we apply \cite[Lemma~4.8]{macroCLT_complex} for the test functions $f_1(x^{-1/2}z_1)$ and $f_2(x^{-1/2}z_2)$, and obtain the desired identity.
\end{proof}

We conclude this section by the proof of Corollary~\ref{cor:general_testfn}.

\begin{proof}[Proof of Corollary~\ref{cor:general_testfn}] We focus on the proof of \eqref{eq:cor_general_testfn}, while for higher order moments the argument is analogous. For $k\in[N]$ denote $x_k:=k/N$. The unrestricted sums over $k, l$ below run over $[N]$, though non-zero contribution comes only from $[c_0N,N]$. Applying \eqref{eq:main_macro} for $n=2$ and test functions $F_1(x_{k},\cdot)$, $F_2(x_{l},\cdot)$, we get
\begin{equation}
N^{-2}\Cov\left(L_N(F_1),L_N(F_2)\right) = N^{-2}\sum_{k,l} C_\beta \Big(\big(x_{k}, F_1(x_{k},\cdot)\big),\big(x_{l}, F_1(x_{l},\cdot)\big)\Big) + \mathcal{O}(N^{-c}).
\label{eq:3d_to_2d}
\end{equation}
The term (together with the factor $N^{-2}$) corresponding to $k,l$ in the rhs. of \eqref{eq:3d_to_2d}, equals to
\begin{equation}
\begin{split}
&\int_{x_k-1/N}^{x_k}\!\!\!\dif t_1\int_{x_l-1/N}^{x_l}\!\!\!\dif t_2  C_\beta \Big(\big(t_1, F_1(x_k,\cdot)\big),\big(t_2, F_1(x_l,\cdot)\big)\Big)+N^{-2}\mathcal{O}\left(N^{-1/2+\varepsilon}\right)\\
&=\int_{x_k-1/N}^{x_k}\!\!\!\dif t_1\int_{x_l-1/N}^{x_l}\!\!\!\dif t_2  C_\beta \Big(\big(t_1, F_1(t_1,\cdot)\big),\big(t_2, F_1(t_2,\cdot)\big)\Big)+N^{-2}\mathcal{O}\left(N^{-1/2+\epsilon}+N^{-\alpha}\right),
\end{split}
\label{eq:C_reg_application}
\end{equation}
for any fixed $\varepsilon>0$. Here we first applied Lemma~\ref{lem:C_reg} and then used that $F_1,F_2$ and $\alpha$-H{\"o}lder regular in the first variable. Summing up \eqref{eq:C_reg_application} over $k,l$, we finish the proof of \eqref{eq:cor_general_testfn}.

\end{proof}

\subsection{Two-body stability analysis}\label{sec:beta} In this section we collect the proofs of Lemma~\ref{lem:beta}, Lemma~\ref{lem:beta_perturb} and \eqref{eq:A_bound}. We often use the following asymptotics of $\rho_x^z$ from \cite[Eq.(3.15)]{spectral_radius}, which we restate here for convenience. For $0<\eta\lesssim 1$ it holds that 
\begin{equation}
\rho^z_x(\ii\eta)\sim \begin{cases}
(x-|z|^2)^{1/2}+\eta^{1/3},&|z|\le\sqrt{x},\\
\frac{\eta}{|x-|z|^2| + \eta^{2/3}},&|z|>\sqrt{x}. 
\end{cases}
\label{eq:rho_asymp}
\end{equation}

\begin{proof}[Proof of Lemma~\ref{lem:beta}] Consider the product of $\beta_\pm$ defined in \eqref{eq:beta_expl}.  A basic manipulation with \eqref{eq:m_cubic} shows that it can be written in the form
\begin{equation}
\beta_+\beta_- = x_1u_1u_2 |z_1-z_2|^2 + (1-x_1u_1)(1-x_1u_2) +x_1|\Im m_1|^2u_2 \frac{1-x_1u_1}{u_1} + x_1|\Im m_2|^2u_1 \frac{1-x_1u_2}{u_2},
\label{eq:beta_prod_id}
\end{equation}
for more details see \cite[Lemma~6.1]{macroCLT_complex}. Each of the terms in the rhs. of \eqref{eq:beta_prod_id} is non-negative. Since $u_1,u_2\sim 1$ by \eqref{eq:def_u} and \eqref{eq:rho_asymp}, the first term in the rhs. of \eqref{eq:beta_prod_id} is of order $|z_1-z_2|^2$. For the third term we have
\begin{equation}
x_1|\Im m_1|^2u_2 \frac{1-x_1u_1}{u_1} \sim |\Im m_1|^2 \frac{1-x_1u_1}{u_1} =|\Im m_1|^2 \frac{\eta_1}{\eta_1+x_1|\Im m_1|}\sim \rho_1\eta_1, 
\end{equation}
where in the last step we used that $\rho_1\gtrsim \eta_1$ by \eqref{eq:rho_asymp}. Similarly, for the fourth term in the rhs. of \eqref{eq:beta_prod_id} it holds that
\begin{equation}
x_1|\Im m_2|^2u_1 \frac{1-x_1u_2}{u_2} = x_1|\Im m_2|^2u_1 \frac{1-x_2u_2}{u_2} + x_1(x_2-x_1) |\Im m_2|^2 u_1\sim \rho_2\eta_2 + \rho_2^2(x_2-x_1).
\end{equation}
This finishes the proof of \eqref{eq:min_beta}.

To prove \eqref{eq:min_beta_outside}, consider $|z_j|\ge \sqrt{x_j}$ for $j=1,2$. We estimate the second term in the rhs. of \eqref{eq:beta_prod_id} from below as follows:
\begin{equation}
(1-x_1u_1)(1-x_1u_2) \ge (1-x_1u_1)(1-x_2u_2) \sim \frac{\eta_1\eta_2}{(\eta_1+\rho_1)(\eta_2+\rho_2)}\gtrsim \left(|z_1|^2-x_1\right)\left(|z_2|^2-x_2\right),
\end{equation}
where in the last bound we used the second line of \eqref{eq:rho_asymp}. This finishes the proof of Lemma~\ref{lem:beta}.
\end{proof}

\begin{proof}[Proof of Lemma~\ref{lem:beta_perturb}] The explicit formula \eqref{eq:beta_expl} implies that to prove \eqref{eq:beta_perturb} it suffices to show
\begin{equation}
\left\lVert \widehat{M}_j-M_j\right\rVert\lesssim \frac{|E_j|}{\eta_j},\quad M_j:=M_{x_j}^{z_j}(\ii\eta_j),\quad \widehat{M}_j:=M_{x_j}^{z_j}(E_j+\ii\eta_j),
\label{eq:M_perturb}
\end{equation} 
for $j=1,2$. Here we treat both $M_j$ and $\widehat{M}_j$ as $2\times 2$ matrices and thus omit in $M$-terms the superscript $(k)$ which indicates the matrix size. We also use the same convention for $\mathcal{S}_{x_j}$. Subtracting \eqref{eq:MDE} for $w=E_j+\ii\eta_j$ and $w=\ii\eta_j$ from each other, we obtain
\begin{equation}
\widehat{M}_j-M_j=E_j\left(I-\widehat{M}_j\mathcal{S}_{x_j}[\cdot]M_j\right)^{-1}[\widehat{M}_jM_j],
\label{eq:M_difference}
\end{equation}
where $I$ acts as identity on $\C^{2\times 2}$. Using that $\|M_j\|,\|\widehat{M}_j\|\lesssim 1$ by \eqref{eq:Mmu_bound} and employing the stability bound from \cite[Lemma~6.1]{macroCLT_real}, we bound the rhs. of \eqref{eq:M_difference} in operator norm from above by $|E_j|/\eta_j$ and finish the proof of Lemma~\ref{lem:beta_perturb}.
\end{proof}

\begin{proof}[Proof of Lemma~\ref{lem:trajectories}] For simplicity, we consider only the case $x=1$, while the general case can be reduced to this one by the rescaling \eqref{eq:M_rescaling}. Solving \eqref{eq:char_flow} explicitly, we get 
\begin{equation}
w_t =\ee^{(T-t)/2}w_T +\left(\ee^{(T-t)/2}-\ee^{-(T-t)/2}\right) m_T,\quad t\in [0,T].
\label{eq:wt_eplicit}
\end{equation}
Therefore, $|\Re w_t|\lesssim |\Re w_T|+|\Re m_T|$. We further have
\begin{equation}
\Re m_T = \Re m^{z_T}(w_T) = \Re \left[ m^{z_T}(w_T)- m^{z_T}(\ii\Im w_T)\right] =\mathcal{O}\left(|\Re w_T|\right) =\mathcal{O}\left(\eta_T\right).
\label{eq:Re_m}
\end{equation}
Here we omitted the subscript $x=1$ in $m^{z_T}$ for brevity. In the second step in \eqref{eq:Re_m} we used that $\Re m^z(w)=0$ for $\Re w=0$, and in the third step that $m^z(w)$ is Lipshitz in $w$ with $\Im w>0$ in the bulk regime. Combining the inputs presented above and using that $\eta_T<\eta_t$, we finish the proof of Lemma~\ref{lem:trajectories}.
\end{proof}

\begin{proof}[Proof of Lemma~\ref{lem:M2_bound}] We start with the proof of \eqref{eq:M2_bound1}. It holds that
\begin{equation}
\begin{split}
x_1\left\vert \langle G_1B_1\widetilde{G}_2B_2\rangle\right\vert =& (2N)^{-1} \left\vert \mathrm{Tr}\left[ G_1B_1P^*G_2PB_2\right]\right\vert\\ 
\le &(2N)^{-1}\left(\mathrm{Tr}\left[ G_1^*G_1B_1P^*PB_1^*\right]\right)^{1/2}\left( \mathrm{Tr}\left[ G_1^*G_2 PB_2B_1^*P^*\right]\right)^{1/2}\\
\le &\|B_1\|\|B_2\|\sqrt{x_1x_2}\langle G_1^*G_1\rangle^{1/2}\langle G_2^*G_2\rangle^{1/2} = \|B_1\|\|B_2\|\sqrt{\frac{x_1|\langle \Im G_1\rangle|}{\eta_1}\frac{x_2|\langle \Im G_2\rangle|}{\eta_2}}.
\end{split}
\label{eq:M2_bound1_proof}
\end{equation}
Here in the first bound we used the Cauchy-Schwarz inequality, and in the last step the Ward identity. Observe that the lhs. of \eqref{eq:M2_bound1_proof} is the deterministic approximation to the lhs. of \eqref{eq:M2_bound1}, and similarly for the rhs. of \eqref{eq:M2_bound1_proof} and \eqref{eq:M2_bound1}. Therefore, \eqref{eq:M2_bound1} follows from \eqref{eq:M2_bound1_proof} by a standard \emph{tensorization argument} (also known as \emph{meta argument}), see e.g. \cite[Section~2.6]{Najim16}, \cite[Proof of Lemma~D.1]{non-herm_overlaps} and \cite[Proof of Lemma~S2.6]{hyperuniformity}. 

Now we prove \eqref{eq:M2_bound2}. It suffices to establish this bound for $\Im w_1\Im w_2>0$, since in the complementary case the sign of $\Im w_1$ can be switched by the elementary identity
\begin{equation}
E_-^{(k_1)}G_{x_1}^{(k_1),z_1}(w_1)E_-^{(k_1)} = - G_{x_1}^{(k_1),z_1}(-w_1),
\label{eq:w_switch}
\end{equation}
for the proof see e.g. \cite[Eq.(7.52)]{hyperuniformity}. We further focus on the bound for the second term in the lhs. of \eqref{eq:M2_bound2}, while for the first term the argument is analogous. It follows from \eqref{eq:P_inverse} that
\begin{equation}
\left\langle M_{1\widetilde{2}}^BE_+\right\rangle = \frac{1}{\beta_+\beta_-} \left(\left(1+x_1\langle M_1E_-\widetilde{M}_2E_-\rangle\right)\langle M_1B \widetilde{M}_2\rangle - x_1 \langle M_1E_-\widetilde{M}_2\rangle \langle M_1B\widetilde{M}_2E_-\rangle\right),
\end{equation}
\label{eq:M2_bound2_proof}
where we additionally used that $\det \mathcal{P}=\beta_+\beta_-$ by \eqref{eq:beta_expl}. An explicit calculation based on \eqref{eq:def_u} and \eqref{eq:M_tilde} shows that
\begin{equation}
\langle M_1E_-\widetilde{M}_2\rangle = \ii\Im [\overline{z}_1z_2] u_1u_2 = \ii\Im [\overline{z}_1(z_2-z_1)]u_1u_2 = \mathcal{O}(|z_1-z_2|).
\end{equation}
Next, using the rescaling \eqref{eq:M_rescaling} we get that $1+x_1\langle M_1E_-\widetilde{M}_2E_-\rangle$ equals to
\begin{equation}
\begin{split}
&\frac{\sqrt{x_2}-\sqrt{x_1}}{\sqrt{x_2}} + \sqrt{\frac{x_1}{x_2}}\left(1+\left\langle M_1^{(k_1),z_1/\sqrt{x_1}}(w_1/\sqrt{x_1}) E_- M_1^{(k_1),z_2/\sqrt{x_2}}(w_2/\sqrt{x_2})E_-\right\rangle\right)\\
&\quad \lesssim \frac{|x_1-x_2|+|z_1-z_2|+|w_1|+|w_2|}{\min\{\rho_1,\rho_2\}^4}.
\end{split}
\label{eq:MEME_bound}
\end{equation}
To go from the first to the second line we used the calculation of $1+x_1\langle M_1E_-\widetilde{M}_2E_-\rangle$ in the case $x_1=x_2=1$ from \cite[Eq.(A.22)]{hyperuniformity}.

Combining \eqref{eq:M2_bound2_proof}--\eqref{eq:MEME_bound} we get
\begin{equation}
\left\vert\left\langle M_{1\widetilde{2}}^BE_+\right\rangle \right\vert \lesssim \left(|x_1-x_2|+|z_1-z_2|+|w_1|+|w_2|\right)\|B\|\frac{1}{|\beta_+\beta_-|}.
\label{eq:M2_prefin}
\end{equation}
From the bound on the imaginary axis \eqref{eq:min_beta} and the condition $|\Re w_j|\le c\eta_j$, for a sufficiently small $c>0$, we further have $|\beta_+\beta_-|\gtrsim \rho_1\eta_1+\rho_2\eta_2$. Together with \eqref{eq:M2_prefin} this finishes the proof of \eqref{eq:M2_bound2}.
\end{proof}

\begin{proof}[Proof of \eqref{eq:A_bound}] Using the explicit formula for $A$ from \eqref{eq:def_A} and recalling that $M$ has the $2\times 2$ block structure \eqref{eq:def_u}, we get
\begin{equation}
A=\left(1-x_1\left(m_1^2 +|z_1|^2u_1^2\right)\right)^{-1}M_1 = \left(\frac{w_1}{w_1+xm_1} - 2x_1m_1^2\right)^{-1}M_1.
\label{eq:A_proof}
\end{equation}
In the second identity we additionally used \eqref{eq:m_cubic}. Since $\Re m_1=0$, there is no cancellation between the two terms in the denominator of the rhs. of \eqref{eq:A_proof}. Using \eqref{eq:rho_asymp} to compute this denominator and estimating $\|M_1\|\lesssim 1$ by \eqref{eq:Mmu_bound}, we complete the proof of \eqref{eq:A_bound}.
\end{proof}

\subsection{Proof of Lemma~\ref{lem:determ_truncation}}\label{sec:truncation} In this proof, we use $\xi>0$ to denote a fixed exponent which may be taken arbitrarily small at the end. The precise value of $\xi$ may change from line to line. Lemma~\ref{lem:Girko_reduction} together with the definition of $I_c^{(k)}(f)$ from \eqref{eq:def_I_c} imply that it suffices to check that
\begin{equation}
\frac{k_1k_2}{(2\pi)^2} \int_\C \dif^2 z_1\int_\C \dif^2 z_2 \overline{\Delta f_1(z_1)}\Delta f_2(z_2) \int_{\eta_c}^{T}\dif\eta_1\int_{\eta_c}^{T}\dif \eta_2 \Cov\left(\langle G_{x_1}^{(k_1),z_1}(\ii\eta_1)\rangle,\langle G_{x_2}^{(k_2),z_2}(\ii\eta_2)\rangle\right) 
\label{eq:truncation1}
\end{equation} 
equals to the rhs. of \eqref{eq:determ_truncation}. First, we remove the edge regime in \eqref{eq:truncation1} where at least one of the parameters $z_j$ is close to $\partial\DD_{x_j}$. To this end, fix a small $\omega>0$ and consider the set of all $z_1,z_2\in \C$ such that $|z_j-\sqrt{x_j}|\le N^{-\omega}$ for some $j=1,2$. The contribution of this regime to \eqref{eq:truncation1} is bounded from above by
\begin{equation}
N^{2-\omega} \int_{\eta_c}^T\dif\eta_1 \int_{\eta_c}^T \dif \eta_2 \frac{N^\xi}{N^2\eta_1\eta_2} \lesssim N^{-\omega+\xi}.
\label{eq:truncation2}
\end{equation}
Here we estimated the covariance in \eqref{eq:truncation1} in absolute value by the single-resolvent local law \eqref{eq:1G_av}. The factor $N^{-\omega}$ comes from the volume of the domain of integration.

Next, consider the complementary regime, where $|z_j-\sqrt{x_j}|>N^{-\omega}$ for both $j=1$ and $j=2$. Denote for short  $G_j:=G_{x_j}^{(k_j),z_j}(\ii\eta_j)$ and observe that by the chiral symmetry of the spectrum of $H^{(k_j),z_j}$, $\overline{\langle G_j\rangle}=-\langle G_j\rangle$. Therefore,
\begin{equation}
\Cov\left(\langle G_1\rangle, \langle G_2\rangle\right) = -\E\langle G_1-\E G_1\rangle\langle G_2-\E G_2\rangle.
\label{eq:Cov_to_E}
\end{equation}
We apply \eqref{eq:Cov_UV} to the rhs. of \eqref{eq:Cov_to_E} and plug this formula into \eqref{eq:truncation1} (where the edge regime has been already removed). The integral of the error term in the rhs. of \eqref{eq:Cov_UV} has an upper bound of order 
\begin{equation}
N^2\int_\C \dif^2 z_1\int_\C \dif^2 z_2 |\Delta f_1(z_1)||\Delta f_2(z_2)|\int_{\eta_c}^T\dif\eta_1\int_{\eta_c}^T \dif \eta_2 \frac{N^{-\delta_c+D\omega+\xi}}{N^2\eta_1\eta_2}\lesssim N^ {-\delta_c+D\omega+\xi}\prod_{j=1}^2\|\Delta f_j\|_{L^2(\Omega)}.
\label{eq:truncation3}
\end{equation}
Here the factor $N^{-\delta_c}$ comes from the factor in the error term, and $N^{D\omega}$ arises from the $\theta$-dependence of the implicit constant in \eqref{eq:Cov_UV} with $\theta:=N^{-\omega}$. We further take $\omega<\delta_c/(2D)$, so that both \eqref{eq:truncation2} and \eqref{eq:truncation3} can be incorporated into the error term in the rhs. of \eqref{eq:determ_truncation}.

In such a way, we obtain that \eqref{eq:determ_truncation} holds with the only difference that the integration in the rhs. of \eqref{eq:determ_truncation} now runs over $\eta_j\in [\eta_c,T]$ and $|z_j-\sqrt{x_j}|>N^{-\omega}$, for $j=1,2$. It remains to show that the missing regimes are negligible. For the integral of $U_1U_2$ the argument is identical to the one in the proof of \cite[Lemma~4.7]{macroCLT_complex} and thus is omitted. To estimate the integral of $V_{12}$ in the regime $|z_1-\sqrt{x_1}|\le N^{-\omega}$, $z_2\in\C$, $\eta_1,\eta_2\in (0,+\infty)$, we first integrate $V_{12}$ over $\eta_1,\eta_2$ using \eqref{eq:V_to_Theta}. Then we use \eqref{eq:Theta_int_bound} with $\mathcal{A}_1$ being the set of $z_1\in\C$ such that $|z_1-\sqrt{x_1}|\le N^{-\omega}$, and obtain the upper bound of order $N^{-\omega/2}\|\Delta f_1\|_{L^2(\Omega)}\|\Delta f_2\|_{L^2(\Omega)}$. The integral in the regime $|z_2-\sqrt{x_2}|\le N^{-\omega}$ is estimated in the same way.

To estimate the tails of $\eta$-integration in the regime $|z_j-\sqrt{x_j}|>N^{-\omega}$, we show that
\begin{equation}
|V_{12}| \lesssim \frac{1}{\left\vert \beta_+\beta_-\right\vert^2}\prod_{j=1}^2\frac{1}{\left\vert |z_j|-\sqrt{x_j}\right\vert(\eta_j+1)^2},
\label{eq:V_bound}
\end{equation}
where $\beta_\pm$ are defined in \eqref{eq:beta_expl}. The proof of \eqref{eq:V_bound} is postponed to the end of this section. Consider first the case $|z_j|<\sqrt{x_j}-N^{-\omega}$, $j=1,2$. Combining \eqref{eq:V_bound} with \eqref{eq:min_beta}, we obtain
\begin{equation}
\int_0^{+\infty}|V_{12}|\dif \eta_2 \lesssim \frac{N^{3\omega}}{\left(|z_1-z_2|^2+\rho_1\eta_1\right)\left(\eta_1+1\right)^2}.
\label{eq:abs_V_int1}
\end{equation}
While \eqref{eq:V_bound} is stated only for $\eta_j\le L$, for some large $L>0$, it is in fact uniform in $\eta_j$, since e.g. for $\eta_2>L$ we have $|\beta_+\beta_-|\sim 1\sim \eta_2\rho_2$. Here the first equivalence follows from the discussion above Lemma~\ref{lem:beta} and the second from \eqref{eq:rho_asymp}. In \eqref{eq:abs_V_int1} we also used that $\rho_2\gtrsim N^{-\omega/2}$ for $\eta_2\lesssim 1$ by \eqref{eq:rho_asymp}. Next, we integrate \eqref{eq:abs_V_int1} over $\eta_1$:
\begin{equation}
\left(\int_0^{\eta_c}+\int_T^{+\infty}\right)\dif \eta_1 \int_0^{+\infty}\dif\eta_2 |V_{12}|\lesssim N^{4\omega} \left(\log \left(1+\frac{\eta_c}{|z_1-z_2|^2}\right) + T^{-1}\right).
\end{equation}
Using this bound and arguing similarly to \eqref{eq:Theta_int_bound} and \eqref{eq:K_difference}--\eqref{eq:sup_L2_bound}, we get
\begin{equation}
\begin{split}
&\int_\C\dif^2 z_1\int_\C\dif^2 z_2 \prod_{j=1}^2\left(\bm{1}\left\lbrace|z_j|\le \sqrt{x_j}-N^{-\omega}\right\rbrace |\Delta f_j(z_j)|\right)\left(\int_0^{\eta_c}+\int_T^{+\infty}\right)\dif \eta_1\int_0^{+\infty}\dif\eta_2 |V_{12}|\\
&\quad\lesssim N^{4\omega}\eta_c^{1/2-\varepsilon} \prod_{j=1}^2 \|\Delta f_j\|_{L^2(\Omega)},
\end{split}
\label{eq:abs_V_int2}
\end{equation}
for any fixed $\varepsilon>0$. Since $\eta_c=N^{-1+\delta_c}$ for a small $\delta_c>0$, the rhs. of \eqref{eq:abs_V_int2} can be incorporated into the error term in the rhs. of \eqref{eq:determ_truncation}, provided that $\omega>0$ is chosen sufficiently small. The bound on the tails of $\eta_2$-integration is identical to the one which we have just presented.

In the remaining $(z_1,z_2)$-regimes the argument is simpler. If $|z_i|<\sqrt{x_i}-N^{-\omega}$ and $|z_j|>\sqrt{x_j}+N^{-\omega}$ for $\{i,j\}=\{1,2\}$, \eqref{eq:min_beta} implies $|\beta_+\beta_-|\ge |z_1-z_2|^2\ge N^{-2\omega}$, which in combination with \eqref{eq:V_bound} gives the desired bound. Finally, in the case when $|z_j|>\sqrt{x_j}+N^{-\omega}$ for $j=1,2$, we use \eqref{eq:min_beta_outside} instead of \eqref{eq:min_beta}.

To complete the proof of Lemma~\ref{lem:determ_truncation}, it remains to verify \eqref{eq:V_bound}. We have from \eqref{eq:def_UV} and \eqref{eq:beta_expl} that
\begin{equation}
2V_{12} =\partial_{\eta_1}\partial_{\eta_2} \log \beta_+\beta_- = \frac{\partial_{\eta_1}\partial_{\eta_2} \left[\beta_+\beta_-\right]}{\beta_+\beta_-} - \frac{\partial_{\eta_1}\left[\beta_+\beta_-\right]\partial_{\eta_2}\left[\beta_+\beta_-\right]}{\left(\beta_+\beta_-\right)^2}.
\label{eq:log_beta_dif}
\end{equation}
To upper bound the derivatives of $\beta_+\beta_-$, we write this quantity in the form
\begin{equation}
\beta_+\beta_- = \left(1-x_1\langle M_1\widetilde{M}_2\right)\left(1+x_1\langle M_1E_-\widetilde{M}_2E_-\rangle\right) + x_1^2 \langle M_1\widetilde{M}_2E_-\rangle \langle M_1E_-\widetilde{M}_2\rangle,
\end{equation}
where we recalled the notations from \eqref{eq:M_tilde}, interpreted $M_1,\widetilde{M}_2$ as $2\times 2$ matrices and denoted by $E_-\in \C^{2\times 2}$ the diagonal matrix with $1,-1$ on the diagonal. We further observe that $\|M_1\|\lesssim (\eta_1+1)^{-1}$, $\|\widetilde{M}_2\|\lesssim (\eta_2+1)^{-1}$, as it follows from \eqref{eq:Mmu_bound} and \eqref{eq:MDE}. In combination with \eqref{eq:A_proof} and the second part of \eqref{eq:MA} this gives
\begin{equation}
\|\partial_{\eta_1}M_1\|\lesssim \left\vert |z_1|-\sqrt{x_1}\right\vert^{-1}\left(\eta_1+1\right)^{-2},\quad \|\partial_{\eta_2}\widetilde{M}_2\|\lesssim \left\vert |z_2|-\sqrt{x_2}\right\vert^{-1}\left(\eta_2+1\right)^{-2}.
\end{equation}
Together with \eqref{eq:log_beta_dif} this finishes the proof of \eqref{eq:V_bound}.\hfill\qed

\subsection{Proof of Wick's Theorem}\label{sec:Wick}

In this section we prove Theorem~\ref{theo:main} in full generality for any joint moments of centered linear statistics. Recall that in Section~\ref{sec:CLTproofs} we have already established this result for the second order moments. For simplicity of presentation we focus on the complex case. The real case requires only minor adjustments similar to the ones in Section~\ref{sec:Cov_real} and thus is omitted.

From Proposition~\ref{prop:underline_comp} we have that
\begin{equation}
\label{eq:covwick}
\E\langle G_1-\E G_1\rangle\langle G_2-\E G_2\rangle=\mathcal{M}\big((x_1,\eta_1,z_1),(x_2,\eta_2,z_2)\big)+\mathcal{O}\left(\frac{N^\xi}{\sqrt{N\eta_*}(N^2\eta_1\eta_2)}\right),
\end{equation}
where $\mathcal{M}$ stands for the rhs. of \eqref{eq:Cov_UV} (without the error term):
\begin{equation}
\mathcal{M}\big((x_1,\eta_1,z_1),(x_2,\eta_2,z_2)\big):= \frac{1}{2x_1x_2N^2}\left(V_{12}+\kappa_4 x_1^2 U_1U_2\right),\quad \text{for}\quad x_1\le x_2,
\label{eq:M_UV_def}
\end{equation}
and for $x_1>x_2$ the variables $(x_1,\eta_1,z_1)$ and $(x_2,\eta_2,z_2)$ in the rhs. of  \eqref{eq:M_UV_def} are interchanged. We point out that the error term in \eqref{eq:Cov_UV} is actually better than the one in \eqref{eq:covwick}. Here, and throughout this section, we use this worse error for notational shortness and simplicity, as it will not affect our final result. We now show that the $\langle G_i-\E G_i\rangle$ are asymptotically jointly Gaussian.
\begin{proposition}[Wick's Theorem]\label{prop:Wick} Assume the set-up and conditions of Proposition~\ref{prop:underline_comp} with the only differences that we additionally fix $n\in\N$, consider $j\in[n]$ instead of $j\in[2]$ and do not assume that $x_1\le x_2$. It holds that
\begin{equation}
\label{eq:wickprod}
\E\prod_{i=1}^n\langle G_i-\E G_i\rangle=\sum_{P\in \Pi_n}\prod_{\{i,j\}\in P}\!\!\mathcal{M}_{ij}+\mathcal{O}_\theta\left(\frac{N^{\xi}}{\sqrt{N\eta_*}}\prod_{i=1}^n \frac{1}{N\eta_i} \right),\quad\mathcal{M}_{ij}:=\mathcal{M}\big((x_i,\eta_i,z_i),(x_j,\eta_j,z_j)\big),
\end{equation}
uniformly in the same sense as explained below \eqref{eq:Cov_UV}. Here $\Pi_n$ denotes the set of all pairings of $[n]$. We use the convention that $\Pi_n=\emptyset$ if $n$ is odd.
\end{proposition}

Using Proposition~\ref{prop:Wick} we conclude the proof of Theorem~\ref{theo:main} in the same way as \eqref{eq:Cov_aim} was derived from Proposition~\ref{prop:underline_comp} in Section~\ref{sec:CLTproofs}. Thus, it remains to prove Proposition~\ref{prop:Wick}.

\begin{proof}[Proof of Proposition~\ref{prop:Wick}]
To keep the notation simpler, we assume that $\theta\sim 1$, which is equivalent to the assumption $||z_i|-\sqrt{x_i}|\gtrsim 1$, $i\in[n]$. It can be readily seen from the proof that all bounds will deteriorate only by a factor $\theta^{-D}$ for some large fixed $D>0$, when considering $||z_i|-\sqrt{x_i}|\ge \theta$.

Applying \eqref{eq:G_underline} to $G_i$ and using the single-resolvent local law \eqref{eq:1G_av} to estimate the first term in the rhs. of \eqref{eq:G_underline}, we get
\begin{equation}
\langle G_i-\E G_i\rangle=-\langle\underline{W_iG_i}A_i\rangle+\E\langle\underline{W_iG_i}A_i\rangle+\mathcal{O}\left(\frac{N^\xi}{(N\eta_i)^2}\right),\qquad A_i:=\frac{M_i}{1-x_i\langle M_i^2\rangle}.
\label{eq:G_underline_i}
\end{equation}
Here we additionally used that $\|A_i\|\lesssim 1$ in our regime of interest by \eqref{eq:A_bound}, and recalled the notation $W_i:=W^{(k_i)}$. We now proceed by induction on $n$. For $n=1$ \eqref{eq:wickprod} trivially follows from \eqref{eq:1G_av}, while the case $n=2$ is covered by Proposition~\ref{prop:underline_comp}. Now, fix an even $n\in\mathbf{N}$ and assume that \eqref{eq:wickprod} holds if the number of resolvents is at most $n-2$. Set $L\in \{n-1, n\}$ and consider the product $\prod_{i\in [L]} \langle G_i-\E G_i\rangle$. We label the resolvents in such a way that $x_1$ is the smallest of all $\{x_i\}_{i\in[n]}$. Applying \eqref{eq:G_underline_i} to the first factor and performing the cumulant expansion in $\langle\underline{W_1G_1}A_1\rangle$ similarly to \eqref{eq:Cov_cum_exp}, we get that this product equals to
\begin{equation}
\label{eq:cumexpw}
\E\sum_{ab}\!\left(\!\! -\frac{1}{N}\langle \Delta^{ab}G_1A_1\rangle\partial_{ab}\!-\!\sum_{l=2}^R\sum_{{\bm \alpha}\in\{ab,ba\}^l}\!\frac{\kappa(ab,{\bm \alpha})}{l! N^{(l+1)/2}}\partial_{{\bm\alpha}}(\langle \Delta^{ab}G_1A_1\rangle\,\cdot\,)+\E\langle\underline{W_1G_1}A_1\rangle\right)\!\!\prod_{i=2}^L\langle G_i-\E G_i\rangle,
\end{equation}
up to the additive error term $\mathcal{O}(N^{-D})$, for any fixed $D>0$ and for sufficiently large fixed $R\in\N$. In \eqref{eq:cumexpw} the first summation runs over $a,b\in\mathcal{I}_1$ defined in \eqref{eq:def_IJ}. Additionally, for any function $f(W)$ in \eqref{eq:cumexpw} we introduced the notation
\[
\partial_{{\bm\alpha}}(\langle \Delta^{ab}G_1A_1\rangle\,\cdot\,)f(W):=\partial_{{\bm\alpha}}\big(\langle \Delta^{ab}G_1A_1\rangle f(W)\big).
\]
The first term in \eqref{eq:cumexpw} comes from the second order cumulant contribution in the cumulant expansion. Note that the derivative $\partial_{ab}$ hits only the product of $\langle G_i-\E G_i\rangle$, $i\in [2,L]$, but not $G_1$.  Meanwhile, the terms in the $l$-summation come from higher order cumulants, and the derivatives there may hit any $G_i$ including $G_1$. As we will shortly see, after all differentiations in \eqref{eq:cumexpw} are performed, the second order cumulant term gives rise to the quantities of the form $V_{1j}$ (see the second line of \eqref{eq:def_UV} with 2 replaced by~$j$), and the fourth order terms ($l=3$) yield $U_1U_j$, the rest is negligible.

It can be easily seen (see e.g. \cite[Section 6]{macroCLT_complex}) that all the terms in the $l$-summation in \eqref{eq:cumexpw} with $l\ne 3$ are bounded by
\[
\mathcal{E}_{[L]}:=\frac{N^\xi}{\sqrt{N\eta_*}}\prod_{i=1}^L\frac{1}{N\eta_i},
\]
and so negligible. This is a consequence of the fact that in all these terms it appears at least an off-diagonal term of the resolvent which is small by the isotropic local law in Theorem~\ref{theo:1Gllaw}. For the same reason the only term for $l=3$ which is not negligible is when all the resolvent entries are diagonal. We are thus left with (see \eqref{eq:usrel} for the first term below and \eqref{eq:cum_exp_kappa4} for the second)
\begin{equation}
\begin{split}
\label{eq:longexpr}
\eqref{eq:cumexpw}=&-\frac{\ii}{4N^2}\E\sum_{p=2}^L \frac{1}{x_p}\left(\sum_{\sigma\in\{\pm\}} \sigma\partial_{\eta_p}\left\langle G_1A_1E_\sigma^{(k_1)} \widetilde{G}_pE_\sigma^{(k_1)}\right\rangle\right)\E\prod_{i=2 \atop i\ne p}^L\langle G_i-\E G_i\rangle \\
&-\frac{1}{2x_1N}\sum_{(a,b)\in\mathcal{J}_1} \frac{\kappa_4}{N^2} \E \partial_{ba} [(G_1A_1)_{ba}]\partial_{ab}\partial_{ba} \left[ \langle G_p\rangle -\E\langle G_p\rangle\right] \E\prod_{i=2 \atop i\ne p}^L\langle G_i-\E G_i\rangle+\mathcal{O}\left(\mathcal{E}_{[L]}\right).
\end{split}
\end{equation}
Here we denoted $\widetilde{G}_p:=P_p^*G_pP_p$, where $P_p$ is defined as in \eqref{eq:defP}, but with $k_2$ replaced by $k_p$. The index set $\mathcal{J}_1$ is defined in \eqref{eq:def_IJ}. In \eqref{eq:longexpr} we also used that the term for $l=3$ when all the three derivatives hit $\langle \Delta^{ab}G_1A_1\rangle$ (which consists of all diagonal resolvent entries) is cancelled by the analogous term in cumulant expansion of $\E\langle\underline{W_1G_1}A_1\rangle$. In fact, all the terms when the derivatives hit only $\langle \Delta^{ab}G_1A_1\rangle$ are cancelled by their counterpart in $\E\langle\underline{W_1G_1}A_1\rangle$.

Finally, proceeding analogously to the computations of Section~\ref{sec:compcov}, we compute that the rhs. of \eqref{eq:longexpr} equals to
\begin{equation}
\sum_{p=2}^n \mathcal{M}_{1p}\,\E\prod_{i=2 \atop i\ne p}^L\langle G_i-\E G_i\rangle+\mathcal{O}\left(\mathcal{E}_{[L]}\right).
\label{eq:Wick_fin}
\end{equation}
Since the number of resolvents in \eqref{eq:Wick_fin} is at most $n-2$, we use the inductive hypothesis and finish the proof of \eqref{eq:wickprod}.
\end{proof}

\bibliographystyle{plain} 
\bibliography{refs}

\end{document}